\documentclass[12pt,reqno]{amsart}

\usepackage{amsmath,amsfonts,amsthm}
\usepackage{mathtools}
\allowdisplaybreaks
\newtheorem{theorem}{Theorem}[section]

\usepackage[a4paper,
top=27mm,bottom=27mm,left=23mm,right=23mm
]{geometry}
\usepackage{graphicx}
\usepackage{booktabs}
\usepackage{relsize}
\usepackage{makecell}
\usepackage{mathrsfs}

\usepackage[swashQ,
]{cochineal}
\usepackage[varqu,varl,scaled=.93]{inconsolata}
\usepackage[cochineal,vvarbb]{newtxmath}
\usepackage{tikz}
\usetikzlibrary{cd}
\usetikzlibrary{arrows.meta}
\usetikzlibrary{matrix,arrows,decorations.pathmorphing}
\tikzcdset{arrow style=tikz,
  diagrams={>=Straight Barb}
}
\usetikzlibrary{positioning}
\usetikzlibrary{decorations.markings}
\tikzset{->-/.style={decoration={
      markings,
      mark=at position #1 with {\arrow{>}}}, postaction={decorate}}}
\tikzset{->>-/.style={decoration={
      markings,
      mark=at position #1 with {\arrow{>>}}}, postaction={decorate}}}

\usepackage[numbers,sort]{natbib}
\usepackage[
bookmarks=true,
bookmarksdepth=subsection,
bookmarksnumbered=true,
breaklinks=true,
colorlinks=true,
linkcolor=blue,
citecolor=magenta
]{hyperref}

\numberwithin{equation}{section}
\newcommand{\I}{\mathrm{i}}

\DeclareMathOperator{\sh}{sh}
\DeclareMathOperator{\ch}{ch}
\DeclareMathOperator{\End}{End}
\DeclareMathOperator{\Mod}{Mod}
\DeclareMathOperator{\Sk}{Sk}
\DeclareMathOperator{\Rel}{Rel}

\DeclareMathDelimiter{\Norm}{\mathord}{largesymbols}{"3E}{largesymbols}{"3E}

\begin{document}

\title{A Note on
  Double Affine Hecke Algebra for
  Skein Algebra
  on Twice-Punctured Torus}

\author[K. Hikami]{Kazuhiro Hikami}
\address{Faculty of Mathematics,
  Kyushu University,
  Fukuoka 819-0395, Japan.}

\email{
  \texttt{khikami@gmail.com}
}
\date{\today}

\begin{abstract}
  We construct a
  generalization of the $C^\vee C_1$-type
  double
  affine Hecke algebra for the skein algebra
  on the twice-punctured torus~$\Sigma_{1,2}$
  using the Heegaard dual of the Iwahori--Hecke operator
  recently introduced in our previous article.
  We show that the automorphisms of our algebra
  correspond to the Dehn twists about the curves on~$\Sigma_{1,2}$.
  We also give the cluster algebraic construction
  of the the classical limit of the skein algebra,
  where
  the Dehn twists are given in terms of the cluster mutations.
  % and
  % we shall
  % reveal  the structure of the
  % coupled Markov equations.
\end{abstract}
\subjclass[2000]{}
\maketitle

\section{Introduction}
The skein algebra  originates from
the Kauffman bracket polynomial invariant for knots,
and it is 
a fundamental tool
of topological quantum field theory
in recent   studies on topological phases of
matter
which 
stem from  the quantum Hall
effect~\cite{Pachos12Book,Simon23Book}.
For instance, we can see that   topological entanglement  entropy
of anyonic quasi-particles
depends on their quantum dimension~\cite{KHikami07c}.
%(see~\cite{KHikami07c} for a computation of the topological entanglement entropy).
One of recent approaches to the skein algebra is the double affine
Hecke algebra (DAHA),
which
was originally introduced to study the
Macdonald polynomial associated to the root
systems~\cite{Chered05Book,Macdonald03book}.
Motivated by the observation
on the relationship between the categorification of knot invariants
and the Macdonald polynomial~\cite{AganaShaki12a},
realized was
that the DAHA is useful for constructing quantum
invariants for links on the once-punctured torus $\Sigma_{1,1}$
and the 4-punctured
sphere $\Sigma_{0,4}$~\cite{IChered13a,IChered16a,GorskNegut15a}.
The relationship between the DAHA and the skein algebra is  
expected  to be useful in the quantized Coulomb
branch~\cite{AlleShan24a}
and in the scattering diagram~\cite{Bouss23a}.

The purpose of this article is to construct a generalization of DAHA
for the skein algebra on the twice-punctured torus $\Sigma_{1,2}$.
In our previous article~\cite{KHikami24a},
proposed was the generalized DAHA for
the skein algebra on $\Sigma_{2,0}$ by introducing
the Heegaard dual of the  Iwahori--Hecke operator of the $C^\vee
C_1$-type DAHA.
We shall make use of the Heegaard dual operator to construct the 
generalized DAHA for the skein algebra on $\Sigma_{1,2}$.
We show that the automorphisms of the generalized DAHA
can be interpreted as
the Dehn twists.
It should be noted~\cite{ChekhShapi22a,FockChek99a} that
the  skein algebra on $\Sigma_{1,2}$
is related
to the
symplectic structure of the upper triangular
matrix~\cite{Bonda04a,Ugagl99a,Boalc01a}
based on the Stokes data of  Frobenius manifolds~\cite{Dubro95}.

% Representation of the skein algebra on surface~$\Sigma$
% is expected to be useful for refinement  of topological structure.
%
% topological entanglement entropy~\cite{KHikami07c}

This paper is organized as follows.
In Section~\ref{sec:mapping} we summarize the skein algebra
$\Sk_A(\Sigma_{1,2})$ and the mapping class group
of the twice-puncture torus.
In Section~\ref{sec:Hecke} is for the DAHA representation.
Based on the isomorphism between the skein algebra
$\Sk_A(\Sigma_{0,4})$ on the 4-punctured sphere and
the $C^\vee C_1$-type DAHA $H_{q,\mathbf{t}}$,
we study $\Sk_A(\Sigma_{1,2})$.
Using the Heegaard dual operator defined in~\cite{KHikami24a}, we
define  $H_{q,\mathbf{t}_\natural}^{gen}$, and show that the automorphism is
consistent with the Dehn twists.
We give explicitly
the $q$-difference operators for variants of the torus knots
$T_{2,2n+1}$ in section~\ref{sec:examples}.
Section~\ref{sec:cluster} is for the classical limit of the skein
algebra.
To see a connection with  a symplectic structure of the upper
triangular matrix, we give explicitly a cluster algebraic description.

Throughout this article, we use the variants of the hyperbolic
functions;
\begin{equation}
  \label{eq:10}
  \ch(x) = x+x^{-1} ,
  \qquad
  \sh(x)= x- x^{-1} .
\end{equation}

\section{Mapping Class Group and Skein Algebra}
\label{sec:mapping}
\subsection{Mapping Class Group}
The mapping class group
$\Mod(\Sigma_{g,n})$ of surface $\Sigma_{g,n}$
is generated by a finite set of the (left-handed) Dehn
twist
$\mathscr{D}_{\mathbb{k}}$
about a simple closed curve $\mathbb{k}$ on $\Sigma_{g,n}$.
When the intersection number of simple closed curves $\mathbb{k}$ and
$\mathbb{k}^\prime$ is zero,
the Dehn twists are commutative,
$\mathscr{D}_{\mathbb{k}} \mathscr{D}_{\mathbb{k}^\prime}
=
\mathscr{D}_{\mathbb{k}^\prime} \mathscr{D}_{\mathbb{k}}$.
See, \emph{e.g.}~\cite{Birman74,FarbMarg11Book}, for  other  properties of the
Dehn twists.

In our  case of the twice-punctured torus~$\Sigma_{1,2}$,
we label the simple closed curves $\mathbb{k}_1$,
$\mathbb{k}_2$,
$\mathbb{k}_3$
and the boundary curves $\mathbb{b}_1$ and $\mathbb{b}_2$ as in
Fig.~\ref{fig:curves}.
Amongst others $\mathbb{x}$ is a separating  curve.
The mapping class group for $\Sigma_{1,2}$  is given by~\cite{ParkeSerie04a}
\begin{equation}
  \label{Mod_12}
  \Mod(\Sigma_{1,2})=
  \left\langle
    \mathscr{D}_1,
    \mathscr{D}_2,
    \mathscr{D}_3    
    ~\middle|~
    \begin{matrix}
      \mathscr{D}_{1,3}=\mathscr{D}_{3,1},
      \mathscr{D}_{1,2,1}=\mathscr{D}_{2,1,2},
      \mathscr{D}_{2,3,2}=\mathscr{D}_{3,2,3},
      \\
      \left(
      \mathscr{D}_{1,2,3}
      \right)^4 =1
    \end{matrix}
  \right\rangle .
\end{equation}
Here
we mean
$\mathscr{D}_i=\mathscr{D}_{\mathbb{k}_i}$,
and
we use
$\mathscr{D}_{i,j,\dots,k}
=\mathscr{D}_i \mathscr{D}_j\dots
\mathscr{D}_k$ for brevity.
Hereafter
the Dehn twisted simple closed curve is
denoted as
\begin{equation}
  \label{define_twist_curve}
  \mathbb{k}_{i, \pm a, \pm b,\dots, \pm c}
  =\mathscr{D}_{c}^{\pm 1}
  \dots
  \mathscr{D}_b^{\pm 1} 
  \mathscr{D}_a^{\pm 1}
  (\mathbb{k}_i).
\end{equation}
We additionally use notations
such as
\begin{equation}
  \mathbb{k}_{i,a^n}=\mathscr{D}_a^n(\mathbb{k}_i) .
\end{equation}
See Fig.~\ref{fig:curves} for some examples.

\begin{figure}[tbhp]
  \centering
  \includegraphics[scale=.7]{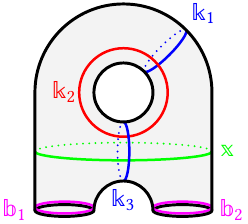}
  \qquad
  \includegraphics[scale=.7]{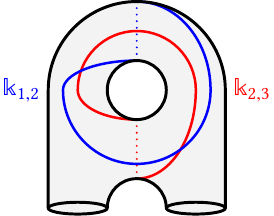}
  \qquad
  \includegraphics[scale=.7]{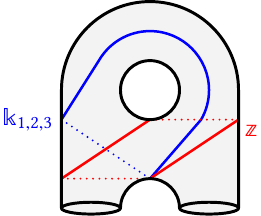}
  \caption{Simple closed curves on $\Sigma_{1,2}$.
    In the left figure,
    shown are  the curves
    $\{\mathbb{k}_1, \mathbb{k}_2, \mathbb{k}_3\}$, the boundary curves
    $\{\mathbb{b}_1, \mathbb{b}_2\}$ of the punctures,
    and the separating curve $\mathbb{x}$.    
    We employ the notation~\eqref{define_twist_curve},
    and
    we denote
    $\mathbb{z}=\mathbb{k}_{1,2,3,-1,-2}$
    in the right figure.  
    % We have $\mathbb{z}=\mathbb{k}_{3,-2,-1,-1,-2}$.
  }
  \label{fig:curves}
\end{figure}

We remark
that,
for
the fundamental group
\begin{equation}
  \label{eq:2}
  \pi_1(\Sigma_{1,2})
  =
  \left\langle
    \mathrm{A}, \mathrm{B}, \mathrm{C}_1, \mathrm{C}_2
    ~\middle|~
    \mathrm{A} \,  \mathrm{B} \,
    \mathrm{A}^{-1} \mathrm{B}^{-1}
    \mathrm{C}_1 \mathrm{C}_2 =1
  \right\rangle ,
\end{equation}
the actions of
the  Dehn twists $\mathscr{D}_i$ 
correspond to the automorphisms~$\delta_i$ defined by~\cite{ParkeSerie04a}
\begin{equation}
  \label{eq:35}
  \delta_1:
  \begin{pmatrix}
    \mathrm{A}\\ \mathrm{B}\\
    \mathrm{C}_1 \\ \mathrm{C}_2
  \end{pmatrix}
  \mapsto
  \begin{pmatrix}
    \mathrm{A}\\ \mathrm{B} \, \mathrm{A}^{-1}\\
    \mathrm{C}_1 \\ \mathrm{C}_2
  \end{pmatrix}
  ,
  \qquad
  \delta_2:
  \begin{pmatrix}
    \mathrm{A}\\ \mathrm{B}\\
    \mathrm{C}_1 \\ \mathrm{C}_2
  \end{pmatrix}
  \mapsto
  \begin{pmatrix}
    \mathrm{A} \,  \mathrm{B}\\ \mathrm{B} \\
    \mathrm{C}_1 \\ \mathrm{C}_2
  \end{pmatrix}
  ,
  \qquad
  \delta_3:
  \begin{pmatrix}
    \mathrm{A}\\ \mathrm{B}\\
    \mathrm{C}_1 \\ \mathrm{C}_2
  \end{pmatrix}
  \mapsto
  \begin{pmatrix}
    \mathrm{C}_2 \mathrm{A} \, \mathrm{C}_2^{-1} \\
    \mathrm{B} \,  \mathrm{A}^{-1} \mathrm{C}_2^{-1} \\
    \mathrm{C}_1 \\
    \mathrm{C}_2 \mathrm{A} \,
    \mathrm{C}_2 \mathrm{A}^{-1} \mathrm{C}_2^{-1}
  \end{pmatrix}
  .
\end{equation}

\subsection{Skein Algebra}

The skein algebra $\Sk_A(\Sigma)$ on surface $\Sigma$ is
generated by isotopy classes of framed  links in $\Sigma \times [0,1]$
satisfying
%on surface $\Sigma$ modulo the following relations;
\begin{gather}
  \label{skein_relation}
  % \raisebox{-\ht\strutbox}{
  \vcenter{\hbox{
    \begin{tikzpicture}%[baseline=(current bounding box)]
      % [baseline={([yshift={-\ht\strutbox}]current
      % bounding box.center)},outer sep=0pt, inner sep=0pt]
      \draw [line width=1.2pt](0,1) --( 1,0) ;
      \draw[line width=10pt,white](0,0)--(1,1);
      \draw[line width=1.2pt](0,0)--(1,1);
    \end{tikzpicture}
  }}
  =
  A \,
  \vcenter{\hbox{
%  \raisebox{-\ht\strutbox}{
    \begin{tikzpicture}%[baseline=(current bounding box)]
      \draw [line width=1.2pt] (1,1) to[out=-135,in=135] (1,0);
      \draw [line width=1.2pt] (0,1) to[out=-45,in=45] (0,0);
    \end{tikzpicture}
  }}
  +A^{-1} \,
  % \raisebox{-\ht\strutbox}{
  \vcenter{\hbox{
    \begin{tikzpicture}
      \draw [line width=1.2pt] (1,1) to[out=-135,in=-45] (0,1);
      \draw [line width=1.2pt] (1,0) to[out=135,in=45] (0,0);
    \end{tikzpicture}
  }}
  ,
  \\[2mm]
  % \raisebox{-\ht\strutbox}{
  \vcenter{\hbox{
    \begin{tikzpicture}%[baseline=(current bounding box)]
      \draw [line width=1.2pt] (0,0) circle (0.5) ;
    \end{tikzpicture}
  }}
%  }
  = -A^2 - A^{-2} .
\end{gather}
A multiplication
$\mathbb{x} \, \mathbb{y}$
of links $\mathbb{x}$ and $\mathbb{y}$ on $\Sigma$
means
that $\mathbb{x}$ is vertically above $\mathbb{y}$,
\begin{gather}
  \mathbb{x} \, \mathbb{y}
  =
%  \text{$\mathbb{x}$ is above $\mathbb{y}$}
  \newcolumntype{C}{>{$}c<{$}}
  \begin{tabular}{|C|}
    \hline
    \hphantom{aa}\mathbb{x}\hphantom{aa}
    \\
    \hline
    \mathbb{y}
    \\
    \hline
  \end{tabular}
\end{gather}
% We have
% the skein relations on $\Sigma_{1,2}$ as
% \begin{gather}
%   \label{eq:14}
%   -\mathbb{k}_i^2 \, \mathbb{k}_{i+1}
%   +\left( A^2+A^{-2}\right)\, \mathbb{k}_i\, \mathbb{k}_{i+1}\, \mathbb{k}_i
%   - \mathbb{k}_{i+1} \, \mathbb{k}_i^2
%   =\left( A^2-A^{-2}\right)^2 \,  \mathbb{k}_{i+1} ,
%   \\
%   -\mathbb{k}_{i+1}^2 \, \mathbb{k}_{i}
%   +\left(A^2+A^{-2} \right)\,  \mathbb{k}_{i+1} \, \mathbb{k}_{i}\, \mathbb{k}_{i+1}
%   - \mathbb{k}_{i} \, \mathbb{k}_{i+1}^2
%   =(A^2-A^{-2})^2 \, \mathbb{k}_{i} ,
% \end{gather}
% where $i=1,2$.
% To define the skein algebra $\Sk_A(\Sigma_{1,2})$, we use the Dehn
% twist to denote simple closed curves as in Fig.~\ref{fig:curves}.
% See Fig.~\ref{fig:curves_etc}.
% For an example, we have 
% \begin{equation*}
%   \mathbb{k}_i \, \mathbb{k}_{i+1}
%   = A \, \mathbb{k}_{i,i+1} + A^{-1} \mathbb{k}_{i+1,i} .
% \end{equation*}
In case that two simple closed curves $\mathbb{x}$ and $\mathbb{y}$
intersect exactly once, we have
\begin{equation}
  \label{eq:11}
  \mathscr{D}_{\mathbb{x}}^{\pm 1}(\mathbb{y})
  =
  \frac{1}{A^{\pm 2} - A^{\mp 2}}
  \left(
    A^{\pm 1} \mathbb{x} \, \mathbb{y} -
    A^{\mp 1}\mathbb{y} \, \mathbb{x}
  \right) .
\end{equation}

% \begin{figure}[tbhp]
%   \centering
%   \includegraphics[scale=.7]{torus2punctureFillK_K2prime}
%   \quad
%   % \includegraphics[scale=.5]{torus2punctureFillK_Kz}
%   % \quad
%   \includegraphics[scale=.7]{torus2punctureFillK_K12_K23}
%   % \quad
%   % \includegraphics[scale=.5]{torus2punctureFillK_K21_K32}
%   \caption{Other curves on $\Sigma_{1,2}$.}
%   \label{fig:curves_etc}
% \end{figure}

In our previous paper~\cite{KHikami19a},
we studied  the skein algebra $\Sk_A(\Sigma_{1,2})$ on the
twice-punctured torus.
Combining the
skein algebras on
the sub-surfaces~$\Sigma_{1,1}$ and~$\Sigma_{0,4}$
(see Fig.~\ref{fig:subsurface}),
we have the relations for  the simple closed curves in
Fig.~\ref{fig:curves} as
%the $\Sigma_{0,4}$-type skein relations
\begin{gather}
  \label{kk_11}
  \Rel(\mathbb{k}_1, \mathbb{k}_2, \mathbb{k}_{1,2}),
  \\
  \label{q_character_11}
  \mathbb{x}
  =
  A  \, \mathbb{k}_1 \,  \mathbb{k}_2 \, \mathbb{k}_{1,2}
  -A^2 \, \mathbb{k}_1^2 - A^{-2} \mathbb{k}_2^2
  -A^2 \mathbb{k}_{1,2}^2
  +A^2+A^{-2} ,
  \\
  \label{kk_04_1}
  A^2 \mathbb{x} \,  \mathbb{k}_3 -A^{-2} \mathbb{k}_3 \, \mathbb{x}
  =
  \left(A^4-A^{-4} \right) \mathbb{z} +
  \left( A^2-A^{-2}\right)
  \left( \mathbb{b}_1+\mathbb{b}_2 \right) \, \mathbb{k}_1 ,
  \\
  A^2 \mathbb{k}_3 \,  \mathbb{z} -A^{-2} \mathbb{z} \, \mathbb{k}_3
  =
  \left(A^4-A^{-4} \right) \mathbb{x} +
  \left( A^2-A^{-2}\right)
  \left( \mathbb{b}_1\mathbb{b}_2 +  \mathbb{k}_1^2 \right),
  \\
  A^2 \mathbb{z} \,  \mathbb{x} -A^{-2} \mathbb{x} \, \mathbb{z}
  =
  \left(A^4-A^{-4} \right) \mathbb{k}_3 +
  \left( A^2-A^{-2}\right)
  \left( \mathbb{b}_1+\mathbb{b}_2 \right) \, \mathbb{k}_1 ,
  \\
  \label{q_character_04}
  \begin{multlined}[b]
    A^2
    \mathbb{x} \,  \mathbb{k}_3 \, \mathbb{z}
    =
    A^4 \, \mathbb{x}^2 + A^{-4} \, \mathbb{k}_3^2 + A^4 \, \mathbb{z}^2
    +
    A^2 \left( \mathbb{k}_1^2 + \mathbb{b}_1 \mathbb{b}_2 \right)
    \mathbb{x}
    +
    A^{-2} \left(\mathbb{b}_1+\mathbb{b}_2\right) \mathbb{k}_1
    \mathbb{k}_3
    \\
    +
    A^{2} \left(\mathbb{b}_1+\mathbb{b}_2\right) \mathbb{k}_1
    \mathbb{z}
    +
    2 \mathbb{k}_1^2 + \mathbb{b}_1^2+\mathbb{b}_2^2
    + \mathbb{k}_1^2 \, \mathbb{b}_1 \,  \mathbb{b}_2 -
    \left(A^2+A^{-2}\right)^2 .
  \end{multlined}
\end{gather}
Here
$\mathbb{z}=\mathbb{k}_{1,2,3,-1,-2}$
as in Fig.~\ref{fig:curves}, and
$\Rel(\mathbb{x}_0, \mathbb{x}_1, \mathbb{x}_2) $ denotes
\begin{equation}
  \label{eq:22}
  A \,  \mathbb{x}_i \, \mathbb{x}_{i+1}
  - A^{-1} \mathbb{x}_{i+1} \, \mathbb{x}_i
  =
  \left(A^2- A^{-2}\right) \mathbb{x}_{i+2} ,
\end{equation}
where the subscripts should be read modulo~$3$
for $i=0,1,2$.
See that
two sets of relations,~\eqref{kk_11}--\eqref{q_character_11}
and~\eqref{kk_04_1}--\eqref{q_character_04},
are   from the skein algebras on the sub-surfaces,
$\Sk_A(\Sigma_{1,1})$ and $\Sk_A(\Sigma_{0,4})$, respectively.
Correspondingly~\eqref{q_character_11} and~\eqref{q_character_04}
are from
the quantized character varieties for $\Sigma_{1,1}$ and
$\Sigma_{0,4}$~\cite{Oblom04a,PrzytSikor00a,BerestSamuel16a},
and the former is known as the quantum Markov equation.

\begin{figure}[tbhp]
  \centering
  \includegraphics[scale=.7]{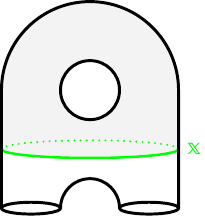}
  \qquad
  \includegraphics[scale=.7]{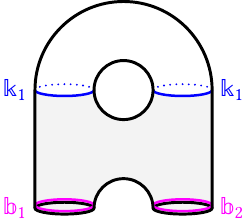}
  \caption{Sub-surfaces of $\Sigma_{1,2}$.
    The left figure indicates that
    the separating curve $\mathbb{x}$ is a boundary curve of
    sub-surface $\Sigma_{1,1}$.
    In the right,
    the sub-surface $\Sigma_{0,4}$ has  boundary curves
    $\mathbb{b}_1$, $\mathbb{b}_2$, and
    $\mathbb{k}_1$.
  }
  \label{fig:subsurface}
\end{figure}

The alternative  presentation of
the skein algebra $\Sk_A(\Sigma_{1,2})$ is as
follows~\cite{Przyty97a,BulloPrzyt99a,FockChek99a};
% (see also~\cite{ChekPenn03a}),
%(see also~\cite{}),
\begin{gather}
  \label{skein_Przytycki_1}
  \Rel(\mathbb{k}_1, \mathbb{k}_2, \mathbb{k}_{1,2}),
  \quad
  \Rel(\mathbb{k}_{1,2}, \mathbb{k}_3, \mathbb{k}_{1,2,3}) 
  \quad
  \Rel(\mathbb{k}_2, \mathbb{k}_3, \mathbb{k}_{2,3}),
  \quad
  \Rel(\mathbb{k}_1, \mathbb{k}_{2,3}, \mathbb{k}_{1,2,3}), 
  \\
  \mathbb{k}_1 \, \mathbb{k}_3=\mathbb{k}_3  \, \mathbb{k}_1 ,
  \qquad
  \mathbb{k}_2  \, \mathbb{k}_{1,2,3}
  = \mathbb{k}_{1,2,3} \, \mathbb{k}_2 ,
  \\
  % \mathbb{k}_{2,3} \,  \mathbb{k}_{1,2}
  % - \mathbb{k}_{1,2} \, \mathbb{k}_{2,3}=
  % \left( A^2- A^{-2}\right)
  % \left(
  %   \mathbb{k}_1\, \mathbb{k}_3 - \mathbb{k}_2 \, \mathbb{k}_{1,2,3}
  % \right),
  % \\
  \mathbb{k}_{2,3} \, \mathbb{k}_{1,2} =
  A^2 \mathbb{k}_1 \, \mathbb{k}_3
  +A^{-2} \mathbb{k}_2 \, \mathbb{k}_{1,2,3}
  + \mathbb{b}_1 +\mathbb{b}_2  ,
  \\
  \label{skein_coupled_Markov}
  \begin{multlined}[b][.9\textwidth]
    \mathbb{k}_1 \, \mathbb{k}_3 \, \mathbb{k}_2 \, \mathbb{k}_{1,2,3}
    +A^4 \mathbb{k}_1^2
    + \mathbb{k}_2^2    + A^{-4}  \mathbb{k}_3^2
    + A^4 \mathbb{k}_{1,2,3}^2
    +A^4 \mathbb{k}_{1,2}^2 +A^{-4} \mathbb{k}_{2,3}^2
    \\
    =
    A^3 \mathbb{k}_1 \, \mathbb{k}_2 \,  \mathbb{k}_{1,2}+
    A^3 \mathbb{k}_{1,2}\,  \mathbb{k}_3 \, \mathbb{k}_{1,2,3}
    + A^{-3} \mathbb{k}_3 \,  \mathbb{k}_2 \, \mathbb{k}_{2,3}
    +A^{-1} \mathbb{k}_1 \, \mathbb{k}_{2,3} \,  \mathbb{k}_{1,2,3}
    \\
    + \mathbb{b}_1  \, \mathbb{b}_2
    +\left( A^2+A^{-2} \right)^2 .
  \end{multlined}
\end{gather}
Here
the last
relation~\eqref{skein_coupled_Markov} is given in a slightly different
expression from the literature
so that
we may regard it as
the quantized  ``coupled'' Markov  triples.
As will be discussed in Section~\ref{sec:cluster},
it is known~\cite{ChekhShapi22a,FockChek99a} that~\eqref{skein_Przytycki_1}--\eqref{skein_coupled_Markov}
can be regarded as a quantization of 
a symplectic groupoid of triangular matrices~\cite{Bonda04a,Ugagl99a,Boalc01a}.

\section{Double Affine Hecke Algebra}
\label{sec:Hecke}
\subsection{Spherical DAHA}

In~\cite{KHikami19a} constructed
was the map
\begin{equation}
  \label{map_A_End}
  \mathcal{A}: \Sk_{A=q^{-\frac{1}{4}}}(\Sigma_{1,2}) \to
  \End
  \mathbb{C}\left(q^{\frac{1}{4}},x_0, b_1,b_2 \right)
  \left[x+x^{-1}\right] ,
\end{equation}
by combining
the double affine Hecke algebras of type-$A_1$ and
type-$C^\vee C_1$.
We have for the simple closed curves in Fig.~\ref{fig:curves} as
follows;
\begin{subequations}
  \label{A_k_x_b}
\begin{align}
  \label{A_k1}
  \mathcal{A}(\mathbb{k}_1)
  &=\ch(x_0),
  \\
  \label{A_k2}
  \mathcal{A}(\mathbb{k}_2)
  &  = \I \, q^{-\frac{1}{4}} \, G_0(x_0 ; x),
  \\
  \label{A_k3}
  \mathcal{A}(\mathbb{k}_3) 
  & =
    \sum_{\epsilon=\pm}
    W(x^\epsilon;x_0,b_1,b_2)
    \left(\eth^\epsilon -1 \right)
    % +
    % W(x^{-1};x_0,b_1,b_2)
    % \left(\eth^{-1} -1 \right)
    -
    \ch \left( q^{-\frac{1}{2}} \, b_1 \, x_0\right) ,
  \\
  \mathcal{A}(\mathbb{x})
  & = \ch(x),
  \\
  \mathcal{A}(\mathbb{b}_1)
  & =\ch (b_1) ,
  \\
  \mathcal{A}(\mathbb{b}_2)
  & =\ch (b_2) .
\end{align}
\end{subequations}
Here and hereafter $\eth$ is the $q$-shift operator for $x$,
\begin{equation}
  \label{eq:30}
    (\eth^{\pm 1} f)(x, x_0)= f(q^{\pm 1}\, x, x_0),
\end{equation}
and
\eqref{A_k3} is nothing but the reduced  Askey--Wilson operator whose
potential is given by
\begin{equation}
  \label{eq:19}
  W(x; x_0,b_1,b_2)=
  -\frac{
    \left(b_1 \, b_2 +q^{\frac{1}{2}} x\right)
    \left(b_1+b_2 \, q^{\frac{1}{2}} x \right)
    \left(q^{\frac{1}{2}} x+x_0^2\right)}{
    b_1 \,  b_2 \,  q^{\frac{1}{2}} \left(1-x^2 \right)
    \left(1-q^{\frac{1}{2}} x\right) x_0} .
\end{equation}
In~\eqref{A_k2}, we have used
the $q$-difference operators for $x_0$ defined by~\cite{KHikami24a}
\begin{equation}
  \label{def_G_op}
  G_n(x_0 ; x)
  =
  \frac{-x_0^{-n}}{1-x_0^{2}} \, \eth_0
  +
  \frac{
    x_0^n
    \left(q^{\frac{1}{2}} x + x_0^{2}\right)
    \left(q^{\frac{1}{2}} +  x \, x_0^{2} \right)}{
    q^{\frac{1}{2}} \, x  \left( 1-x_0^{2} \right)}  \,
  \eth_0^{-1} ,
\end{equation}
where the $q$-shift operator~$\eth_0$ is
\begin{equation}
  \label{eq:8}
  \left( \eth_0^{\pm 1} f \right)(x,x_0)
  =f
  ( x,  q^{\pm \frac{1}{2}} x_0 ) .
\end{equation}
In the following we also use the $q$-difference operator for $x_0$
defined by~\cite{KHikami24a}
\begin{equation}
  \label{def_K_op}
  K_n(x_0; x)
  =
    \frac{-x_0^{-n}}{1-x_0^2} \, \eth_0
    +
    \frac{
    x_0^n
    \left(q^{\frac{1}{2}} x + x_0^{2}\right)
    \left(q^{\frac{3}{2}} x+x_0^{2} \right)}{
    q\, x  \left( 1-x_0^{2} \right)} \, \eth_0^{-1} .
\end{equation}
Messy but plain calculations prove that they satisfy
\begin{gather}
  \label{K_G_G_K}
  K_n(x_0 ; x) \, G_n(x_0 ; q\, x)
  = G_n(x_0 ; x) \, K_n(x_0 ; x),
  \\
  \label{K_K-G_G}
  K_n(x_0 ; x^{-1}) \, K_n(x_0 ; q^{-1} x)
  - \left[ G_n(x_0 ; x) \right]^2 =
  -q^{\frac{n}{2}} x^{-1} \left( q^{\frac{1}{2}}-x\right)^2
  ,
  \\
  \label{recursion_K_G}
  \begin{multlined}[b][.8\textwidth]
    \begin{pmatrix}
      K_{n+1}(x_0 ; x^{-1}) & G_{n+1}(x_0 ; x)
    \end{pmatrix}
    \\
    =
    \frac{q^{\frac{1}{2}}}{q^{\frac{1}{2}}-x}
    \begin{pmatrix}
      K_{n}(x_0 ; x^{-1}) & G_{n}(x_0 ; x)
    \end{pmatrix}
    \begin{pmatrix}
      -x \ch(x_0)
      &  - \frac{x+q^{\frac{1}{2}}x_0^2}{x_0}
      \\
      \frac{q^{\frac{1}{2}}+x x_0^2}{x_0}
      &
        q^{\frac{1}{2}} \ch(x_0)
    \end{pmatrix} .
  \end{multlined}
\end{gather}

As was shown in~\cite{KHikami19a},
the skein relations~\eqref{kk_11}--\eqref{q_character_04} are
fulfilled with~\eqref{A_k_x_b} and
\begin{subequations}
  \begin{align}
    \label{A_k12}
    \mathcal{A}(\mathbb{k}_{1,2})
    & = \I \, G_{-1} (x_0 ; x) ,
    \\
    \label{A_z}
    \mathcal{A}(\mathbb{z})
    & =
      \sum_{\epsilon=\pm}
      W(x^\epsilon; x_0, b_1, b_2) \, \left(
      q^{\frac{1}{2}} x^\epsilon \eth^\epsilon -1
      \right)
      - \ch \left(q^{-\frac{1}{2}} b_1 \, x_0 \right) .
      % \frac{q^{\frac{1}{2}}\left( b_1+b_2 \right)
      % \left( 1+b_1 b_2\right)  \left( 1+x_0^2\right) \, x
      % }{
      %   b_1 \,  b_2 \left(
      %     q^{\frac{1}{2}}-x \right)
      %   \left(1-q^{\frac{1}{2}} x\right)}
      %   .
  \end{align}
  For the
  relations~\eqref{skein_Przytycki_1}--\eqref{skein_coupled_Markov},
  we need
  \begin{align}
    \label{A_k23}
    &
      \begin{multlined}[b][.8\textwidth]
        \mathcal{A}(\mathbb{k}_{2,3})
        =
        \I  \sum_{\epsilon=\pm}
        W(x^\epsilon;x_0,b_1,b_2)
          \frac{ x_0 \, x^\epsilon }{q^{\frac{1}{2}}x^\epsilon+x_0^2}
          K_1(x_0; x^\epsilon) \, \eth^\epsilon
        \\
        +\I \,
        \frac{(b_1+b_2) (1+b_1 b_2) x}{
          b_1 b_2 \left( q^{\frac{1}{2}}-x \right)
          \left( 1-q^{\frac{1}{2}}x \right)} \,
        G_1(x_0 ; x)
        ,
      \end{multlined}
    \\
    \label{A_k123}
    &
      \begin{multlined}[b][.8\textwidth]
        \mathcal{A}(\mathbb{k}_{1,2,3})
        =
        \I \sum_{\epsilon=\pm}
        W(x^\epsilon; x_0, b_1, b_2)
        \frac{q^{\frac{1}{4}} x_0 \,  x^\epsilon}{q^{\frac{1}{2}}
          x^\epsilon+x_0^2}
        K_0(x_0; x^\epsilon)
        \eth^\epsilon
        \\
        +\I \,
        \frac{q^{\frac{1}{4}}
          \left( b_1+b_2\right)
          \left(1+b_1 b_2\right) \, x}{
          b_1 b_2 \left(q^{\frac{1}{2}}-x\right)
          \left(1-q^{\frac{1}{2}}x\right)
        } \,
        G_0(x_0 ; x) .
      \end{multlined}
  \end{align}
\end{subequations}

We note that
\begin{equation}
  \label{k_21n}
  \mathcal{A}(\mathbb{k}_{2,1^n})
  = \I \, q^{-\frac{n+1}{4}}G_n(x_0 ; x) ,
\end{equation}
and that the skein relations
\begin{subequations}
  \label{skein_1-2_1n}
  \begin{gather}
    \mathbb{k}_1 \mathbb{k}_{2,1^n}
    =A \, \mathbb{k}_{2,1^{n-1}}+ A^{-1} \mathbb{k}_{2,1^{n+1}},
    \\
    \mathbb{k}_{2,1^n}\mathbb{k}_1
    =A^{-1} \, \mathbb{k}_{2,1^{n-1}}+ A \, \mathbb{k}_{2,1^{n+1}},
  \end{gather}
\end{subequations}
correspond to
\begin{subequations}
  \begin{gather}
    G_n(x_0 ; x) \, \ch(x_0)
    =
    q^{\frac{1}{2}} G_{n-1}(x_0 ; x) + q^{-\frac{1}{2}}G_{n+1}(x_0 ; x),
    \\
    \label{x0_G_n}
    \ch(x_0) \, G_n(x_0 ; x)
    =
    G_{n-1}(x_0 ; x) + G_{n+1}(x_0 ; x) .
  \end{gather}
\end{subequations}

\subsection{Generalized DAHA}
The DAHA of $C^\vee C_1$-type
$H_{q,\mathbf{t}}=
H_{q,t_0,t_1,t_2,t_3}$ is generated by
$\mathsf{T}_0$, $\mathsf{T}_1$,
$\mathsf{T}_0^\vee=q^{-\frac{1}{2}} \mathsf{T}_0^{-1}\mathsf{X}$, and
$\mathsf{T}_1^\vee=\mathsf{X}^{-1} \mathsf{T}_1^{-1}$,
satisfying the Iwahori--Hecke relations
(see, \emph{e.g.},~\cite{Macdonald03book,NoumiStokm00a}),
\begin{equation}
  \label{Hecke_original}
  \begin{alignedat}{2}
    &
      (\mathsf{T}_0 - t_0^{-1}) (\mathsf{T}_0+t_0)=0,
    &\qquad \qquad
    &
      (\mathsf{T}_0^\vee - t_2^{-1}) (\mathsf{T}_0^\vee+t_2)=0 ,
      \\
    &
      (\mathsf{T}_1 - t_1^{-1}) (\mathsf{T}_1+t_1)=0,
    &\qquad \qquad
    &
      (\mathsf{T}_1^\vee - t_3^{-1}) (\mathsf{T}_1^\vee+t_3)=0 .
  \end{alignedat}
\end{equation}
These are  the Hecke operators for the Askey--Wilson
polynomials~\cite{NoumiStokm00a}, and
we have the homomorphism from $\Sk_{A=q^{-\frac{1}{4}}}(\Sigma_{0,4})$~\cite{Oblom04a}.

% We follow the notation of~\cite{KHikami19a,KHikami24a}, and
We fix the parameters~$\mathbf{t}$ as~\cite{KHikami19a}
\begin{equation}
  \label{t_natural}
  \mathbf{t}_\natural=
  \left(
    \I  \, x_0, \I \,  q^{-\frac{1}{2}} b_1,
    \I  \, x_0, \I \, b_2
  \right) ,
\end{equation}
and the polynomial representations of the Iwahori--Hecke operators are given by
\begin{equation}
  \label{eq:27}
  \begin{aligned}
    \mathsf{T}_0
    & \mapsto
      \I \frac{x}{q^{\frac{1}{2}} - x}
      \left(
      -\frac{q^{\frac{1}{2}}+x_0^2 \, x}{x_0 \, x} \,
      \mathsf{s} \, \eth + x_0+x_0^{-1}
      \right),
    \\
    \mathsf{T}_1
    & \mapsto
      \I \left(
      -\frac{q^{\frac{1}{2}}}{b_1} +
      \frac{
      \left(b_1 \,  b_2+q^{\frac{1}{2}} x \right)
      \left( b_1+q^{\frac{1}{2}}b_2 \, x\right)
      }{q^{\frac{1}{2}} b_1 \, b_2
      \left( 1-x^2\right)
      }
      \left( \mathsf{s} -1\right)
      \right) ,
    \\
    \mathsf{X}
    & \mapsto x,
  \end{aligned}
\end{equation}
where
\begin{equation}
  \label{eq:4}
  \mathsf{s} \, x=x^{-1} \, \mathsf{s} .
\end{equation}
%(see also~\cite{Macdonald03book,Chered05Book}).
Due to that the sub-surface $\Sigma_{0,4}$ in Fig.~\ref{fig:subsurface}
has two identical boundary monodromies corresponding
to~$\mathbb{k}_1$,
the parameters $t_0$ and $t_2$ are identified in $\mathbf{t}_\natural$~\eqref{t_natural}.
%the operator $\mathsf{T}_0$ includes only  one parameter $x_0$
%different from
%in the original definition in~\cite{NoumiStokm00a}.
We see that
the Iwahori--Hecke relations~\eqref{Hecke_original} reduce to
\begin{equation}
  \label{T0_Hecke}
  \begin{aligned}
    &
      \begin{aligned}[b]
        \sh\left( \mathsf{T}_0\right)
        % \mathsf{T}_0 - \mathsf{T}_0^{-1}
        & =
          -\I \ch(x_0)
        \\
        &=
          \sh \left(
          q^{-\frac{1}{2}} \mathsf{T}_0^{-1} \mathsf{X}
          \right)
          =
          \sh \left(
          q^{-\frac{1}{2}} \mathsf{X} \mathsf{T}_0^{-1}
          \right)
          % q^{-\frac{1}{2}} \mathsf{T}_0^{-1} \mathsf{X}
          % -
          % \left( q^{-\frac{1}{2}} \mathsf{T}_0^{-1} \mathsf{X} \right)^{-1}
          % =
          % -q^{\frac{1}{2}} \mathsf{T}_0 \mathsf{X}^{-1}
          % +q^{-\frac{1}{2}} \mathsf{X} \mathsf{T}_0^{-1}
          ,
      \end{aligned}
    \\
    &
      \sh \left( \mathsf{T}_1 \right)
%  \mathsf{T}_1 - \mathsf{T}_1^{-1}
      = -\I \ch (q^{-\frac{1}{2}} b_1) ,
    \\
    &
      \sh \left( \mathsf{T}_1 \mathsf{X} \right)
      = \I \ch(b_2) .
  % \mathsf{X}^{-1} \mathsf{T}_1^{-1}
  % -\left(  \mathsf{X}^{-1} \mathsf{T}_1^{-1}\right)^{-1}
  % =
  % -\I \ch(b_2) .
  \end{aligned}
\end{equation}
% The Iwahori--Hecke operator $\mathsf{T}_1$ 
% reduces to 
% the operator in~\cite{KHikami24a} when $b_1=b_2$. 

Following~\cite{KHikami24a}, we define the Heegaard dual $\mathsf{U}_0$
of the
Iwahori--Hecke operator~$\mathsf{T}_0$ by
\begin{equation}
  \label{define_U0}
  \mathsf{U}_0
  \mapsto
  \frac{q^{-\frac{1}{4}} x}{q^{\frac{1}{2}} - x}
  K_0(x_0;  x^{-1}) \, \mathsf{s} \, \eth
  -
  \frac{q^{-\frac{1}{4}} x}{q^{\frac{1}{2}} - x}
  G_0(x_0 ;  x) .
\end{equation}
We find with a help of~\eqref{K_G_G_K}--\eqref{K_K-G_G} that
this satisfies  Hecke-type relations
\begin{gather}
  \label{U0_Hecke}
  \begin{aligned}[b]
    \sh \left(\mathsf{U}_0\right)
    &   =
      q^{-\frac{1}{4}} G_0(x_0; x)
    \\
    &  =
      \sh\left(q^{-\frac{1}{2}}  \mathsf{U}_0^{-1} \mathsf{X}
      \right)
      =
      \sh\left(q^{-\frac{1}{2}} \mathsf{X}\, \mathsf{U}_0^{-1} 
      \right) .    
  \end{aligned}
  % \mathsf{U}_0 - \mathsf{U}_0^{-1} 
  % =
  % q^{-\frac{1}{2}}  \mathsf{U}_0^{-1} \mathsf{X}
  % -
  % \left(
  %   q^{-\frac{1}{2}}  \mathsf{U}_0^{-1} \mathsf{X}
  % \right)^{-1}
  % =
  % q^{-\frac{1}{4}} G_0(x_0,x) .
\end{gather}
We can also  see  that~\cite{KHikami24a}
\begin{equation}
  \label{T0U0T0U0X}
  \mathsf{T}_0^{-1}\mathsf{U}_0 \mathsf{T}_0 \mathsf{U}_0^{-1} \mathsf{X}
  = -q   .
\end{equation}
% We note that,
% as a generalization of~\eqref{U0_Hecke}, we have
% \begin{equation}
%   \label{eq:9}
%   \begin{aligned}[b]
%     \mathsf{U}_0
%     \left( \I \, q^{-\frac{1}{4}} \mathsf{T}_0 \right)^{-n}
%     -
%     \left( \I \, q^{-\frac{1}{4}} \mathsf{T}_0 \right)^{n}
%     \mathsf{U}_0^{-1}
%     &  =
%       q^{-\frac{1}{2}}
%       \left( \I \, q^{-\frac{1}{4}} \mathsf{T}_0 \right)^{n}
%       \mathsf{U}_0^{-1} \mathsf{X}
%       -
%       q^{\frac{1}{2}}
%       \mathsf{X}^{-1} \mathsf{U}_0
%       \left( \I \, q^{-\frac{1}{4}} \mathsf{T}_0 \right)^{-n}
%     \\
%     &  =
%       q^{-\frac{n+1}{4}} G_n(x_0,x) .
%   \end{aligned}
% \end{equation}

We denote the generalized DAHA
$H_{q,\mathbf{t}_\natural}^{gen}$
as an addition of
$\mathsf{U}_0$ to $H_{q,\mathbf{t}_\natural}$.
The generalized
spherical DAHA is then
\begin{equation}
  \label{eq:33}
  SH_{q,\mathbf{t}_\natural}^{gen}
  =
  \mathsf{e}  \, H_{q,\mathbf{t}_\natural}^{gen}\, \mathsf{e} ,
\end{equation}
where
the idempotent is
\begin{equation}
  \label{eq:5}
  \mathsf{e}
  =
   \frac{b_1}{q-b_1^2}
  \left(
    - b_1 + \I \, q^{\frac{1}{2}} \mathsf{T}_1
  \right).
\end{equation}
See that
\begin{gather}
  \mathsf{e}^2 = \mathsf{e},
  \\
  \label{T1_and_e}
  \mathsf{T}_1^{\pm 1} \, \mathsf{e}
  =
  \mathsf{e} \, \mathsf{T}_1^{\pm 1}
  =
  \left( -\I \, q^{\frac{1}{2}}b_1^{-1} \right)^{\pm 1} \, \mathsf{e} .
\end{gather}

Our claim on  the map~\eqref{map_A_End} is  as follows.
\begin{theorem}
  \label{thm:map}
  We have
  \begin{equation}
    \mathcal{A}: \Sk_{A=q^{-\frac{1}{4}}}(\Sigma_{1,2})
    \to
    SH_{q,\mathbf{t}_\natural}^{gen} .
  \end{equation}

% We note that the separating curve $\mathbb{x}$ in
% Fig.~\ref{fig:curves} is
% \begin{equation}
%   \label{eq:16}
%   \mathcal{A}(\mathbb{x})=\ch( x).
% \end{equation}

%  \label{thm:auto_T}
  Furthermore 
  we have a commutative diagram for the Dehn twists~\eqref{Mod_12};
  \begin{equation}
    \begin{tikzcd}
      \Sk_{A=q^{-\frac{1}{4}}}(\Sigma_{1,2})\arrow[r, "\mathscr{D}_i"]
      \arrow[d, "\mathcal{A}"']
      & \Sk_{A=q^{-\frac{1}{4}}}(\Sigma_{1,2})
      \arrow[d, "\mathcal{A}"]
      \\
      SH_{q,\mathbf{t}_\natural}^{gen}
      \arrow[r, "\mathscr{T}_i"']
      &
      SH_{q,\mathbf{t}_\natural}^{gen}
    \end{tikzcd}
  \end{equation}
  Here
  $\mathscr{T}_i$ is 
  the  automorphism of $H_{q,\mathbf{t}_\natural}^{gen}$ defined by
  \begin{subequations}
    \begin{align}
      \label{eq:7}
      \mathscr{T}_1:
      &
        \begin{pmatrix}
          \mathsf{T}_0 \\ \mathsf{T}_1 \\ \mathsf{X} \\ \mathsf{U}_0
        \end{pmatrix}
        \mapsto
        \begin{pmatrix}
          \mathsf{T}_0 \\ \mathsf{T}_1 \\ \mathsf{X} \\
          -\I \, q^{\frac{1}{4}} \mathsf{U}_0\mathsf{T}_0^{-1}
        \end{pmatrix} ,
      \\
      \mathscr{T}_2:
      &
        \begin{pmatrix}
          \mathsf{T}_0 \\ \mathsf{T}_1 \\ \mathsf{X} \\ \mathsf{U}_0
        \end{pmatrix}
        \mapsto
        \begin{pmatrix}
          \I \, q^{-\frac{1}{4}} \mathsf{U}_0 \mathsf{T}_0
          \\ \mathsf{T}_1 \\ \mathsf{X} \\ \mathsf{U}_0
        \end{pmatrix} ,
      \\
      \mathscr{T}_3:
      &
        \begin{pmatrix}
          \mathsf{T}_0 \\ \mathsf{T}_1 \\ \mathsf{X} \\ \mathsf{U}_0
        \end{pmatrix}
        \mapsto
        \begin{pmatrix}
          \mathsf{T}_0 \\ \mathsf{T}_1 \\
          \left(\mathsf{T}_0 \mathsf{T}_1\right)^{-1}
          \mathsf{X} \mathsf{T}_1 \mathsf{T}_0
          \\
          q^{-\frac{1}{4}} \left( \mathsf{T}_0 \mathsf{T}_1 \right)^{-1}
          \mathsf{U}_0
        \end{pmatrix} .
    \end{align}
  \end{subequations}
\end{theorem}

\begin{proof}[Proof of Theorem~\ref{thm:map}]
  The first part
  %of Theorem~\ref{thm:map}
  is checked from the fact
  that~\eqref{A_k1}--\eqref{A_k3}
  are written as
  \begin{subequations}
    \begin{align}
      \label{eq:13}
      &\mathcal{A}(\mathbb{k}_1)
        = \ch(\I \, \mathsf{T}_0) ,
      \\
      &
        \mathcal{A}(\mathbb{k}_2)
        = \ch (\I \, \mathsf{U}_0 ) ,
      \\
      &
        \mathcal{A}(\mathbb{k}_3) \, \mathsf{e}
        =
        \ch ( \mathsf{T}_1 \mathsf{T}_0 ) \, \mathsf{e}
        =
        \ch ( \mathsf{T}_0 \mathsf{T}_1 ) \, \mathsf{e}
        ,
    \end{align}
  \end{subequations}
  and~\eqref{A_k12}--\eqref{A_k123} are
  \begin{subequations}
    \begin{align}
      \label{eq:20}
      & \mathcal{A}(\mathbb{k}_{1,2}) \, \mathsf{e}
        =
        \mathscr{T}_2(\mathcal{A}(\mathbb{k}_1)) \, \mathsf{e}
        =
        \ch\left(
        -q^{-\frac{1}{4}} \mathsf{U}_0 \mathsf{T}_0
        \right) \, \mathsf{e},
      \\
      % & \mathcal{A}(\mathbb{k}_{2,1})
      % =
      % \mathscr{T}_1(\mathcal{A}(\mathbb{k}_2))
      % =
      % \ch\left(
      %   q^{\frac{1}{4}} \mathsf{U}_0 \mathsf{T}_0^{-1}
      % \right),
      % \\
  & \mathcal{A}(\mathbb{k}_{2,3}) \, \mathsf{e}
    =
    \mathscr{T}_3 (\mathcal{A}(\mathbb{k}_2)) \, \mathsf{e}
    =
    \ch
    \left(
    \I \, q^{-\frac{1}{4}}
    \left( \mathsf{T}_0 \mathsf{T}_1 \right)^{-1}
    \mathsf{U}_0
    \right) \, \mathsf{e},
      \\
      % & \mathcal{A}(\mathbb{k}_{3,2}) \, \mathsf{e}
      % =
      % \mathscr{T}_2 (\mathcal{A}(\mathbb{k}_3)) \, \mathsf{e}
      % =
      % \ch
      % \left(
      %   \I \, q^{-\frac{1}{4}}
      %   \mathsf{T}_1 \mathsf{U}_0  \mathsf{T}_0
      % \right) \, \mathsf{e},
      % \\
  & \mathcal{A}(\mathbb{k}_{1,2,3}) \, \mathsf{e}
    =\mathscr{T}_3(\mathcal{A}(\mathbb{k}_{1,2})) \, \mathsf{e}
    =\ch\left(
    q^{\frac{1}{2}} \mathsf{T}_1^{-1} \mathsf{X}^{-1} \mathsf{U}_0
    \right) \, \mathsf{e} ,
      \\
      & \mathcal{A}(\mathbb{z}) \, \mathsf{e}
        =
        % \mathscr{T}_{2,3,-1,-2}
        \mathscr{T}_{2}^{-1} \mathscr{T}_1^{-1}(
        \mathcal{A}(\mathbb{k}_{1,2,3})) \, \mathsf{e}
        =
        \ch \left(
        q^{\frac{1}{2}} \mathsf{T}_1^{-1} \mathsf{X}^{-1} \mathsf{T}_0
        \right) \, \mathsf{e} .
    \end{align}
  \end{subequations}

% Our claim is that these automorphisms play roles of
% the Dehn twists of~$\Sigma_{1,2}$.
% As we  have for~\eqref{A_k1}--\eqref{A_k3}
% we can check that~\eqref{A_k12}--\eqref{A_k123} can be written as
% In fact, we have the following.

  For the second part,
  we need to check the corresponding relations to those for the Dehn twists~\eqref{Mod_12}.
  It is straightforward to see the braid relations,
  \begin{equation}
    \mathscr{T}_{1,3}=\mathscr{T}_{3,1},
    \qquad
    \mathscr{T}_{1,2,1}=\mathscr{T}_{2,1,2},
    \qquad
    \mathscr{T}_{2,3,2}=\mathscr{T}_{3,2,3} .
  \end{equation}
  Here as before we mean $\mathscr{T}_{i,j,\dots,k}=\mathscr{T}_i
  \mathscr{T}_j  \dots \mathscr{T}_k$.
  Moreover
  we  find
  \begin{equation}
    \left( \mathscr{T}_{1,2,3} \right)^4:
    \begin{pmatrix}
      \mathsf{T}_0 \\ \mathsf{T}_1 \\ \mathsf{X} \\ \mathsf{U}_0
    \end{pmatrix}
    \mapsto
    \mathsf{T}_1^{-1}
    \begin{pmatrix}
      \mathsf{T}_0 \\ \mathsf{T}_1 \\ \mathsf{X} \\ \mathsf{U}_0
    \end{pmatrix}
    \mathsf{T}_1 ,
  \end{equation}
  which proves
  $
  \left( \mathscr{T}_{1,2,3} \right)^4=1
  $
  on $SH_{q,\mathbf{t}_\natural}^{gen}$ from~\eqref{T1_and_e}.
\end{proof}

% \begin{equation}
%   \label{eq:9}
%   \begin{aligned}[b]
%     \mathsf{U}_0
%     \left( \I \, q^{-\frac{1}{4}} \mathsf{T}_0 \right)^{-n}
%     -
%     \left( \I \, q^{-\frac{1}{4}} \mathsf{T}_0 \right)^{n}
%     \mathsf{U}_0^{-1}
%     &  =
%       q^{-\frac{1}{2}}
%       \left( \I \, q^{-\frac{1}{4}} \mathsf{T}_0 \right)^{n}
%       \mathsf{U}_0^{-1} \mathsf{X}
%       -
%       q^{\frac{1}{2}}
%       \mathsf{X}^{-1} \mathsf{U}_0
%       \left( \I \, q^{-\frac{1}{4}} \mathsf{T}_0 \right)^{-n}
%     \\
%     &  =
%       q^{-\frac{n+1}{4}} G_n(x_0,x) .
%   \end{aligned}
% \end{equation}

We note that the 2-chain relation gives the Dehn twist about the
separating curve $\mathbb{x}$ in Fig.~\ref{fig:curves} by
\begin{equation}
  \label{eq:15}
  \mathscr{D}_{\mathbb{x}}=\left(
    \mathscr{D}_{1,2}\right)^6 ,
\end{equation}
and that 
the corresponding automorphism is
\begin{equation}
  \mathscr{T}_{\mathbb{x}}=\left( \mathscr{T}_{1,2}\right)^6:
  \begin{pmatrix}
    \mathsf{T}_0 \\ \mathsf{T}_1 \\ \mathsf{X} \\ \mathsf{U}_0
  \end{pmatrix}
  \mapsto
  \begin{pmatrix}
    \mathsf{X}  \mathsf{T}_0 \mathsf{X}^{-1} \\
    \mathsf{T}_1 \\ \mathsf{X} \\
    \mathsf{X}\mathsf{U}_0 \mathsf{X}^{-1}
  \end{pmatrix} 
  .
\end{equation}
This is nothing but the automorphism of the $C^\vee C_1$ DAHA studied
in~\cite{IChered16a}.

%%%%%%%
%\section{DAHA Polynomials}
\section{Examples of $q$-Difference Operators}
\label{sec:examples}

We should note that the polynomial representation gives
\begin{equation}
  \label{define_U_n}
  \mathsf{U}_n= \I^{-n} q^{\frac{n}{4}} \, \mathsf{U}_0
  \mathsf{T}_0^{-n} 
  \mapsto
  \frac{q^{-\frac{n+1}{4}} x}{q^{\frac{1}{2}} - x} \,
  K_n(x_0; x^{-1}) \, \mathsf{s} \, \eth
  -
  \frac{q^{-\frac{n+1}{4}} x}{q^{\frac{1}{2}} - x} \,
  G_n(x_0; x) ,
\end{equation}
which can be proved by induction 
using~\eqref{recursion_K_G}.
It is straightforward to see the Hecke-type relation
\begin{equation}
  \label{eq:3}
  \sh\left(   \mathsf{U}_n \right) = q^{-\frac{n+1}{4}}G_n(x_0; x) ,
\end{equation}
and
we find that
the Heegaard dual Iwahori--Hecke operator
satisfy 
\begin{subequations}
  \label{eq:40}
  \begin{gather}
    \mathsf{U}_n  \, \ch(x_0)
    = q^{\frac{1}{4}} \mathsf{U}_{n-1} + q^{-\frac{1}{4}}
    \mathsf{U}_{n+1},
    \\
    \ch(x_0) \,  \mathsf{U}_n
    = q^{-\frac{1}{4}} \mathsf{U}_{n-1} + q^{\frac{1}{4}}
    \mathsf{U}_{n+1} , 
  \end{gather}
\end{subequations}
in accordance with the skein relations~\eqref{skein_1-2_1n}.
%\subsection{$\mathbb{k}_{2,1^n}$, $\mathbb{k}_{1,2^n}$, $\mathbb{k}_{3,2^n}$}
One immediately finds 
that the expression~\eqref{k_21n}  follows from
\begin{equation}
  \label{A_k_21n}
  \begin{aligned}[b]
    \mathcal{A}\left( \mathbb{k}_{2,1^n} \right)
    & = \I \, q^{-\frac{n+1}{4}}G_n(x_0; x) 
    \\
    &=\ch \left( \I \, \mathsf{U}_n \right)
      =
      \ch \left(
      \I  \, \mathsf{U}_0
      \left( \I \, q^{-\frac{1}{4}} \mathsf{T}_0\right)^{-n}
      \right) .
      % =
      % \color{red}{
      % \ch
      % \left( \I \,  q^{-\frac{1}{2}}
      % \left( \I \, q^{-\frac{1}{4}} \mathsf{T}_0 
      % \right)^n \mathsf{U}_0^{-1} \mathsf{X}
      % \right)
      % } .
  \end{aligned}
\end{equation}
% For examples, we have
% \begin{align}
%   \label{eq:37}
%   \mathcal{A}\left(\mathbb{k}_{2,1} \right)
%   &
%     =\ch \left(
%     q^{\frac{1}{4}} \mathsf{U}_0^{-1} \mathsf{T}_0
%     \right)
%     = \I \, q^{-\frac{1}{2}} G_1(x_0,x) ,
%   \\
%   \mathcal{A} \left(\mathbb{k}_{2,-1} \right)
%   & =
%     \ch \left(- q^{-\frac{1}{4}} \mathsf{U}_0 \mathsf{T}_0 \right)
%     =\I \, G_{-1}(x_0, x) .
% \end{align}

For $\mathbb{k}_{1,2^n}$, 
we get
\begin{equation}
  \label{k_1_2n}
  \begin{aligned}[b]
    & \mathcal{A}(\mathbb{k}_{1,2^n})
      = \ch \left( \I
      \left( \I \, q^{-\frac{1}{4}} \mathsf{U}_0 \right)^n
      \mathsf{T}_0
      \right)
    \\
    & =
      \begin{cases}
        \I^n q^{-\frac{n}{4}}
        \left\{

        q^{\frac{1}{4}}
        \widehat{\varphi}_{n-1} \, G_{-1}(x_0; x)
        +\widehat{\varphi}_{n-2} \, \ch(x_0)
        \right\}
        ,
        % \begin{multlined}[b][.65\textwidth]
        %   \I^n
        %   \Bigl\{
        %   G_{-1}(x_0,x) \, g_{n-1}(G_0(x_0,x))
        %   + q^{\frac{1}{2}} G_1(x_0,x) \, g_{n-3}(G_0(x_0,x))
        %   \\
        %   + q^{\frac{n}{4}} \ch(x_0) \,        \delta_{n\in 2\mathbb{Z}}
        %   \Bigr\},
        % \end{multlined}
        & \text{for $n > 0$},
        \\
        \ch(x_0),
        & \text{for $n=0$},
        \\
        \I^{-n} q^{-\frac{n}{4}}
        \left\{
        q^{-\frac{3}{4}}
        \widehat{\varphi}_{-n-1} \,
        G_1(x_0; x)
        +
        \widehat{\varphi}_{-n-2} \,
        \ch(x_0)
        \right\}
        ,
        % \begin{multlined}[b][.65\textwidth]
        %   \I^{|n|} q^{\frac{n}{2}}
        %   \Bigl\{
        %   G_{1}(x_0,x) \, g_{|n|-1}(G_0(x_0,x))
        %   +q^{\frac{1}{2}} G_{-1}(x_0,x) \, g_{|n|-3}(G_0(x_0,x))
        %   \\
        %   +q^{\frac{|n|}{4}}\ch(x_0)  \,   \delta_{n\in 2\mathbb{Z}}
        %   \Bigr\} ,
        % \end{multlined}
        & \text{for $n<0$,}
      \end{cases}
  \end{aligned}
\end{equation}
where
\begin{equation}
  \label{eq:45}
  \widehat{\varphi}_{n}=\varphi_{n}\left(q^{-\frac{1}{4}}
    G_0(x_0; x)\right).
\end{equation}
Here
the $n$-th polynomial $\varphi_n(x)$
is defined by
\begin{equation}
  \label{eq:12}
  \varphi_n(x)=\I^{-n} \, U_n\left( \tfrac{\I}{2}x \right),
\end{equation}
where  the Chebyshev polynomial of the second-kind is
\begin{equation}
  \label{eq:38}
  U_n\left(\tfrac{\ch(x)}{2}\right)
  =
  \frac{x^{n+1}-x^{-n-1}}{x-x^{-1}} .
\end{equation}
It should be noted that  the 3-term relation is
\begin{equation}
  \label{eq:36}
  \begin{gathered}[t]
    \varphi_n(x)= x \, \varphi_{n-1}(x) +\varphi_{n-2}(x),
    \\
    \varphi_0(x)=1,
    \qquad \varphi_1(x)=x .
  \end{gathered}
\end{equation}
The identity~\eqref{k_1_2n} is proved  as follows.
For the case of $n>0$, we have
\begin{equation*}
  \begin{aligned}[b]
    &
    \ch\left( \I^{n+1} q^{-\frac{n}{4}} \mathsf{U}_0^n \mathsf{T}_0
      \right)
      \stackrel{\eqref{T0U0T0U0X}}{=\joinrel=}
      \I^{n+1} q^{-\frac{n}{4}}
      \left\{
      \mathsf{U}_0^n \mathsf{T}_0
      -
      \left( q^{-\frac{1}{2}} \mathsf{U}_0^{-1} \mathsf{X}\right)^n
      \mathsf{T}_0^{-1}
      \right\}
    \\
    &
      \stackrel{\eqref{U0_Hecke}}{=\joinrel=}
      \I^{n+1} q^{-\frac{n}{4}}
      \Bigl\{
      \left(
      \widehat{\varphi}_{n-1}\,
      \mathsf{U}_0
      +\widehat{\varphi}_{n-2} 
      \right) \mathsf{T}_0
      -
      \left(
      \widehat{\varphi}_{n-1} \,
      q^{-\frac{1}{2}}\mathsf{U}_0^{-1} \mathsf{X}
      + \widehat{\varphi}_{n-2}
      \right) \mathsf{T}_0^{-1}
      \Bigr\}    \\
    &
      =
      \I^{n+1}q^{-\frac{n}{4}}
      \left\{
      q^{\frac{1}{4}}
      \widehat{\varphi}_{n-1} \,
        \ch\left( q^{-\frac{1}{4}} \mathsf{U}_0 \mathsf{T}_0 \right)
      +
      \widehat{\varphi}_{n-2} \,
        \sh\left( \mathsf{T}_0\right)
        \right\},
  \end{aligned} 
\end{equation*}
which reduces to~\eqref{k_1_2n} from~\eqref{T0_Hecke}
and~\eqref{A_k_21n}.
The case for $n<0$ follows from the same manner.

For $\mathbb{k}_{3,2^n}$ with  $n>0$, we have 
\begin{equation*}
  \begin{aligned}[b]
    \mathcal{A} \left( \mathbb{k}_{3,2^n} \right)\mathsf{e}
    & =
      \ch \left( \mathsf{T}_1 \left(
      \I \, q^{-\frac{1}{4}} \mathsf{U}_0 \right)^n \mathsf{T}_0
      \right)\mathsf{e}
    \\
    &
      \stackrel{\eqref{T0U0T0U0X}}{=\joinrel=}
      \I^n   q^{-\frac{n}{4}}
      \left(
      \mathsf{T}_1 \mathsf{U}_0^n \mathsf{T}_0
      +
      \left(q^{-\frac{1}{2}} \mathsf{U}_0^{-1} \mathsf{X} \right)^n
      \mathsf{T}_0^{-1} \mathsf{T}_1^{-1}
      \right) \mathsf{e}
    \\
    &
      \stackrel{\eqref{U0_Hecke}}{=\joinrel=}
        \I^n   q^{-\frac{n}{4}}
        \Bigl\{
        \widehat{\varphi}_{n-1}
        \left(
          \mathsf{T}_1 \mathsf{U}_0 \mathsf{T}_0
          + q^{-\frac{1}{2}} \mathsf{U}_0^{-1} \mathsf{X} \mathsf{T}_0^{-1} \mathsf{T}_1^{-1}
        \right)
        +
        \widehat{\varphi}_{n-2}
        \left(
          \mathsf{T}_1 \mathsf{T}_0+
          \mathsf{T}_0^{-1}\mathsf{T}_1^{-1}
        \right)
        \Bigr\}\mathsf{e} ,
  \end{aligned}
\end{equation*}
where we have used
\begin{equation}
  \label{eq:42}
  G_n(x_0; x) \mathsf{T}_1= \mathsf{T}_1\, G_n(x_0; x)
\end{equation}
The last expression can be written in terms of
$\mathcal{A}(\mathbb{k}_3)$~\eqref{A_k3} and
\begin{multline}
  \label{eq:41}
  \mathcal{A}(\mathbb{k}_{3,2})\mathsf{e}
  =
  \I  \sum_{\epsilon=\pm}
  W(x^\epsilon;x_0,b_1,b_2) \,
  \frac{ x_0  }{q^{\frac{1}{2}}x^\epsilon+x_0^2} \,
  K_{-1}(x_0; x^\epsilon) \, \eth^\epsilon
  \\
  +\I \,
  \frac{(b_1+b_2) (1+b_1 b_2) q^{\frac{1}{2}}x}{
    b_1 b_2 \left( q^{\frac{1}{2}}-x \right)
    \left( 1-q^{\frac{1}{2}}x \right)} \,
  G_{-1}(x_0; x)
  .
\end{multline}
As a result, we obtain
\begin{equation}
  \label{eq:39}
  \begin{aligned}[b]
    &
      \mathcal{A} \left( \mathbb{k}_{3,2^n} \right)\mathsf{e}
      % =
      % \ch \left( \mathsf{T}_1 \left(
      % \I \, q^{-\frac{1}{4}} \mathsf{U}_0 \right)^n \mathsf{T}_0
      % \right)\mathsf{e}
    \\
    & =
      \begin{cases}
        \left(
        \I^{n-1}  q^{-\frac{n-1}{4}}
        \widehat{\varphi}_{n-1}\,
        \mathcal{A}\left(\mathbb{k}_{3,2}\right)
        +
        \I^n  q^{-\frac{n}{4}}
        \widehat{\varphi}_{n-2} \,
        \mathcal{A}\left(\mathbb{k}_{3}\right)
        \right)\mathsf{e}
        ,
        & \text{for $n>0$,}
        \\
        \mathcal{A}(\mathbb{k}_3) \, \mathsf{e},
        & \text{for $n=0$},
        \\
        \left(
        \I^{-n-1} q^{-\frac{n+1}{4}}
        \widehat{\varphi}_{-n-1} \,
        \mathcal{A}\left(\mathbb{k}_{2,3}\right)
        +
        \I^{-n}  q^{-\frac{n}{4}}
        \widehat{\varphi}_{-n-2} \,
        \mathcal{A}\left(\mathbb{k}_{3}\right)
        \right)
        \mathsf{e}
        ,
        & \text{for $n<0$,}
      \end{cases}
  \end{aligned}
\end{equation}
where $\mathcal{A}(\mathbb{k}_{2,3})$ is given in~\eqref{A_k23}.

\begin{figure}[htbp]
  \centering
  \includegraphics[scale=.7]{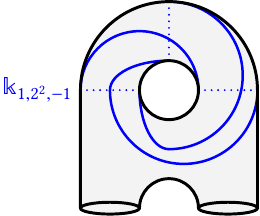}
  \qquad
  \includegraphics[scale=.7]{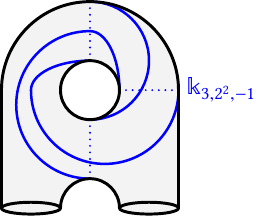}
  \caption{Variants of the trefoil $T_{2,3}$ on $\Sigma_{1,2}$.}
  \label{fig:torus_23}
\end{figure}

We shall  pay attentions to variants of the torus knots $T_{2,2n+1}$ in
Fig.~\ref{fig:torus_23}, whose $N$-colored Jones polynomial is a typical example of the quantum
modular form~\cite{KHikami02c}.
For
$\mathbb{k}_{1,2^2,1^{-n}}$
(see Fig.~\ref{fig:torus_23} for a case of $n=1$),
we obtain by use of~\eqref{define_U_n}
\begin{align}
  &  \mathcal{A}\left(\mathbb{k}_{1,2^2,1^{-n}}\right)
    =
    \ch \left( (-q^{\frac{1}{2}})^{-n} \mathsf{U}_0 \mathsf{T}_0^n
    \mathsf{U}_0 \mathsf{T}_0^{n+1}
    \right)
    =
    \ch\left(-
    q^{-\frac{1}{4}} \mathsf{U}_{-n}\mathsf{U}_{-n-1}
    \right)
    \nonumber
  \\
  &
    \begin{multlined}[b][.8\textwidth]
      =
      \frac{q^{\frac{n}{2}} x}{\left( q^{\frac{1}{2}}-x \right)^2}
      \Bigl\{
      K_{-n}(x_0;  x^{-1}) \,
      K_{-n-1}\left( x_0; q^{-1} x  \right)
      +q^{\frac{1}{2}} K_{-n-1}(x_0 ; x^{-1}) \,
      K_{-n}\left(x_0;  q^{-1} x\right)
      \\
      -
      q^{-\frac{1}{2}}x \, G_{-n}(x_0; x) \, G_{-n-1}(x_0; x)
      -q \, x^{-1} G_{-n-1}(x_0; x) \, G_{-n}(x_0; x)
      \Bigr\} .
    \end{multlined}
    \label{A_k_1221-n}
\end{align}
For $\mathbb{k}_{3,2^2,1^{-n}}$
(see Fig.~\ref{fig:torus_23} for a case of $n=1$), we
have
\begin{align}
  &
    \mathcal{A} \left( \mathbb{k}_{3,2^2,1^{-n}} \right) \mathsf{e}
    =
    \ch\left(
    (-1)^{n+1}q^{-\frac{n+1}{2}}
    \mathsf{U}_0 \mathsf{T}_0^n \mathsf{U}_0 \mathsf{T}_0^{n+1} \mathsf{T}_1
    \right)
    =
    \ch \left(
    \I \, q^{-\frac{1}{4}} \mathsf{U}_{-n} \mathsf{U}_{-n-1} \mathsf{T}_1
    \right) \mathsf{e}
    \nonumber
  \\
  &
    \begin{multlined}[b][.7\textwidth]
      =
      \Biggl\{
      \frac{
        \left( b_1 b_2+q^{\frac{1}{2}} x \right)
        \left( b_1+q^{\frac{1}{2}} b_2 x\right)
      }{b_1 b_2 (1-x^2)} \,
      \frac{q^{\frac{n}{2}} x}{
        \left(1-q^{\frac{1}{2}}x \right)^2}
      \\
      \times
      \left\{
        q^{-\frac{1}{2}} x^{-1} K_{-n-1}(x_0; x) \, G_{-n} (x_0; q\, x)
        -G_{-n-1}(x_0; x) \, K_{-n}(x_0; x)
      \right\} \eth 
      \\
      -
      \frac{
        \left( b_1 b_2x+q^{\frac{1}{2}}  \right)
        \left( b_1x+q^{\frac{1}{2}} b_2 \right)
      }{b_1 b_2 (1-x^2)}
      \frac{q^{\frac{n}{2}} x}{
        \left(x-q^{\frac{1}{2}} \right)^2}
      \\
      \times
      \left\{
        q^{-\frac{1}{2}} x\, K_{-n-1}(x_0; x^{-1}) \,
        G_{-n} (x_0;  q \, x^{-1})
        -G_{-n-1}(x_0; x) \, K_{-n}(x_0; x^{-1})
      \right\} \eth^{-1}
      \\
      +
      \frac{(b_1+b_2) (1+b_1 b_2)}{b_1 b_2}
      \frac{q^{\frac{n+1}{2}}x^2}{
        \left( q^{\frac{1}{2}}- x \right)^3
        \left(1-q^{\frac{1}{2}} x\right)}
      \\
      \times
      \Bigl\{
      K_{-n}(x_0; x^{-1})  \, K_{-n-1}(x_0; q^{-1}x)
      +q^{\frac{1}{2}} K_{-n-1}(x_0;  x^{-1}) \, K_{-n}(x_0; q^{-1} x)
      \\
      - q^{-\frac{1}{2}}x \, G_{-n}(x_0; x) \, G_{-n-1}(x_0; x)
      - q \,x^{-1}
      G_{-n-1}(x_0; x) \, G_{-n}(x_0; x)
      \Bigr\}
      \Biggr\} \, \mathsf{e}.
    \end{multlined}
    \label{A_k_3221-n}
\end{align}
One sees that $O(\eth^0)$ in~\eqref{A_k_3221-n} coincides
with~\eqref{A_k_1221-n} up to prefactor.

%%%%%%%
%\section{Degenerate DAHA}

%%%%%%%
\section{Classical Skein Algebra and Cluster Algebra}
\label{sec:cluster}

The classical  skein algebra
$\Sk^{cl}(\Sigma)$ generated by simple closed curves on~$\Sigma$
has
the Goldman
Poisson structure
defined
for the Fenchel--Nielsen coordinate system
of the Teichm{\"u}ller space~\cite{VGTura91b}.
In this section we employ the cluster algebra~\cite{FominZelev02a} to study the Poisson structure
of the simple
closed curves (see \emph{e.g.}~\cite{FomiShapThur08a,FominThurs12a}).
The results of this section follow simply from the trace
map~\cite{BonaWong11a}, but we hope
that it is
useful for further studies on the generalized DAHA and the skein
algebra on surface.

% To see a generalization of the Markov equation which originates from
% $\Sk^{cl}(\Sigma_{1,1})$,
% we redefine the
% the simple closed curves in Fig.~\ref{fig:curves} as
% \begin{equation}
%   \begin{psmallmatrix}
%     \mathbb{x}_1\\
%     \mathbb{x}_2\\
%     \mathbb{y}_1\\
%     \mathbb{y}_2\\
%     \mathbb{z}_1\\
%     \mathbb{z}_2
%   \end{psmallmatrix}
%   \sim
%   \begin{psmallmatrix}
%     \mathbb{k}_{1,2,3} \\
%     \mathbb{k}_{2} \\
%     \mathbb{k}_{1} \\
%     \mathbb{k}_{3} \\
%     \mathbb{k}_{2,3} \\
%     \mathbb{k}_{1,2} 
%   \end{psmallmatrix} .
% \end{equation}
% The algebra $\Sk^{cl}(\Sigma_{1,2})$ is
% generated by $\mathbb{x}_a$,
% $\mathbb{y}_a$,
% $\mathbb{z}_a$ ($a=1,2$).

The classical skein algebra
$\Sk^{cl}(\Sigma_{1,2})$
is  given
from~\eqref{skein_Przytycki_1}--\eqref{skein_coupled_Markov} as a
limit $A\to 1$, and the Poisson structure is endowed from
non-commutativity of the skein algebra~\eqref{skein_relation}.
It is generated by a set of curves
$\{
\mathbb{k}_1, \mathbb{k}_2, \mathbb{k}_3,
\mathbb{k}_{1,2}, \mathbb{k}_{2,3}, \mathbb{k}_{1,2,3}
\}$
satisfying
\begin{gather}
  \label{classical_skein_1}
  % \Rel^{cl}(\mathbb{x}_1, \mathbb{y}_1, \mathbb{z}_1),
  % \quad
  % \Rel^{cl}(\mathbb{x}_2, \mathbb{y}_2, \mathbb{z}_1),
  % \quad
  % \Rel^{cl}(\mathbb{y}_2, \mathbb{x}_1, \mathbb{z}_2),
  % \quad
  % \Rel^{cl}(\mathbb{y}_1, \mathbb{x}_2, \mathbb{z}_2),
  \Rel^{cl}(\mathbb{k}_1, \mathbb{k}_2, \mathbb{k}_{1,2}),
  \quad
  \Rel^{cl}(\mathbb{k}_{1,2}, \mathbb{k}_3, \mathbb{k}_{1,2,3}) 
  \quad
  \Rel^{cl}(\mathbb{k}_2, \mathbb{k}_3, \mathbb{k}_{2,3}),
  \quad
  \Rel^{cl}(\mathbb{k}_1, \mathbb{k}_{2,3}, \mathbb{k}_{1,2,3}), 
  \\
  % \{ \mathbb{x}_1~,~ \mathbb{x}_2\} =0,
  % \qquad
  % \{ \mathbb{y}_1 ~,~ \mathbb{y}_2\}=0,
  %\qquad
  %     \{ \mathbb{z}_1 ~,~ \mathbb{z}_2\}
  % = \mathbb{x}_1 \mathbb{x}_2 -  \mathbb{y}_1 \mathbb{y}_2,
  \{ \mathbb{k}_1 ~,~ \mathbb{k}_3\} = 0,
  \qquad
  \{ \mathbb{k}_2 ~,~ \mathbb{k}_{1,2,3}\} = 0,
  \qquad
  \{ \mathbb{k}_{1,2} ~,~ \mathbb{k}_{2,3}\}
  =\mathbb{k}_1 \, \mathbb{k}_3 - \mathbb{k}_2 \, \mathbb{k}_{1,2,3} ,
  \\
  \label{coupled_Markov_0}
  % \mathbb{z}_1 \, \mathbb{z}_2
  % =\mathbb{x}_1 \, \mathbb{x}_2+\mathbb{y}_1 \, \mathbb{y}_2
  % +\mathbb{b}_1+\mathbb{b}_2,
  \mathbb{k}_{1,2} \, \mathbb{k}_{2,3} =
  \mathbb{k}_1 \, \mathbb{k}_3
  + \mathbb{k}_2 \, \mathbb{k}_{1,2,3}
  + \mathbb{b}_1 +\mathbb{b}_2 ,
  \\
  \label{coupled_Markov}
  \begin{multlined}[b][.8\textwidth]
    % \mathbb{x}_1 \mathbb{x}_2 \mathbb{y}_1 \mathbb{y}_2
    % +
    % \mathbb{x}_1^2+\mathbb{x}_2^2 +\mathbb{y}_1^2+\mathbb{y}_2^2
    % + \mathbb{z}_1^2+\mathbb{z}_2^2
    % \\
    % =
    % \left(
    %   \mathbb{x}_1 \mathbb{y}_2 + \mathbb{x}_2 \mathbb{y}_1
    % \right) \mathbb{z}_2
    % + \left(
    %   \mathbb{x}_1 \mathbb{y}_1 + \mathbb{x}_2 \mathbb{y}_2
    % \right)
    % \mathbb{z}_1
    % + \mathbb{b}_1 \mathbb{b}_2 +4 .
    \mathbb{k}_1 \, \mathbb{k}_2 \, \mathbb{k}_3 \, \mathbb{k}_{1,2,3}
    + \mathbb{k}_1^2
    + \mathbb{k}_2^2 
    +   \mathbb{k}_3^2
    + \mathbb{k}_{1,2}^2 + \mathbb{k}_{2,3}^2+  \mathbb{k}_{1,2,3}^2
    \\
    =
    \left(
      \mathbb{k}_1 \, \mathbb{k}_2   +  \mathbb{k}_3 \,
      \mathbb{k}_{1,2,3}
    \right) \,  \mathbb{k}_{1,2}
     +
     \left(
       \mathbb{k}_2 \,  \mathbb{k}_3  + \mathbb{k}_1 \,
       \mathbb{k}_{1,2,3}
     \right) \, \mathbb{k}_{2,3}
    + \mathbb{b}_1  \, \mathbb{b}_2 +4 .
  \end{multlined}
\end{gather}
Here $\Rel^{cl}(\mathbb{x}_0, \mathbb{x}_1, \mathbb{x}_2)$ denotes the
Poisson relations
\begin{equation}
  \label{eq:26}
  \{ \mathbb{x}_i ~,~ \mathbb{x}_{i+1} \}
  =
  \frac{1}{2} \mathbb{x}_i \, \mathbb{x}_{i+1} - \mathbb{x}_{i+2} ,
\end{equation}
where $i=0,1,2$ and the subscripts are read modulo $3$.
See that both the boundary curves,  $\mathbb{b}_1$ and
$\mathbb{b}_2$,
are center,
\emph{i.e.},
they Poisson-commute with other simple closed curves.
% $\mathbb{x}_a$, $\mathbb{y}_a$, and $\mathbb{z}_a$ for $a=1,2$.
We note~\cite{Bonda04a,Ugagl99a,Boalc01a} that the same symplectic
structure was studied for the upper triangular matrix
(see also~\cite{ChekhShapi22a}).
It should be noted
that~\eqref{coupled_Markov_0}--\eqref{coupled_Markov}
can be regarded  as
the coupled Markov equations;
they reduce to
the
Markov equation~\cite{Aigner13Book}
\begin{equation}
  \label{eq:28}
  \mathbb{x}_1^2+\mathbb{y}_1^2+\mathbb{z}_1^2
  =\mathbb{x}_1 \mathbb{y}_1 \mathbb{z}_1
  + 2- \mathbb{b}_1 ,
\end{equation}
when we remove the second puncture $\mathbb{b}_2$ by setting
\begin{equation}
  % \mathbb{x}_1=\mathbb{x}_2,
  % \qquad
  % \mathbb{y}_1=\mathbb{y}_2,
  % \qquad
  % \mathbb{x}_1 \mathbb{y}_1 -\mathbb{z}_1=\mathbb{z}_2,
  \mathbb{x}_1
  = \mathbb{k}_2=\mathbb{k}_{1,2,3} ,
  \qquad
  \mathbb{y}_1
  =\mathbb{k}_1=\mathbb{k}_3 ,
  \qquad
  \mathbb{z}_1
  =\mathbb{k}_{2,3}=\mathbb{k}_1 \mathbb{k}_{1,2,3}-\mathbb{k}_{1,2},
  \qquad
  \mathbb{b}_2=-2 .
\end{equation}
See~\cite{Nakanis21a} for the Thurston coordinates of $\Sigma_{1,2}$,
and   see~\cite{NakanNaata95a} as well for the Diophantine equation
generalizing the Markov equation
associated to $\Sigma_{1,2}$.

%The algebra $\Sk^{cl}(\Sigma_{1,2})$  is endowed with the Poisson structure.

\begin{figure}[tbhp]
  \centering
  \includegraphics[scale=.7]{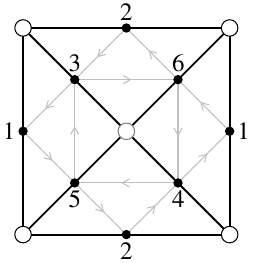}
  \qquad
  \includegraphics[scale=.7]{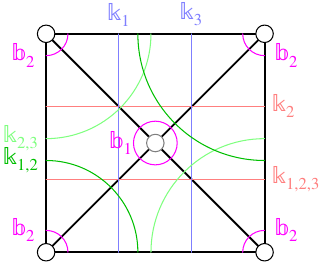}
  \caption{The triangulation and the cluster
    structure associated with $\Sigma_{1,2}$.
    In both figures,
    the square denotes the torus as usual by identifying the north
    (resp. west)
    edge
    with the south
    (resp. east) edge.
    The circles~$\circ$ on corners denote the same puncture,
    while the 
    circle~$\circ$ in the center is another puncture.
    In the left figure,
    the black circles~$\bullet$ denote vertices for the cluster
    variables, which are
    put on each edge of
    ideal triangles.
    The simple closed curves in Fig.~\ref{fig:curves}
    are depicted in the right.}
  \label{fig:triangle_Sigma12}
\end{figure}

The algebra~\eqref{classical_skein_1}--\eqref{coupled_Markov} can be
constructed by the trace map on the cluster variables assigned to each
vertex~$\bullet$ in Fig.~\ref{fig:triangle_Sigma12}.
We introduce the ideal triangulation of $\Sigma_{1,2}$  as in
Fig.~\ref{fig:triangle_Sigma12}.
For each triangle whose vertices are punctures, we  assign  a
dual counter-clockwise oriented triangle to obtain the quiver for $\Sigma_{1,2}$.
We label each vertex in the dual oriented triangles as in Fig.~\ref{fig:triangle_Sigma12},
and we set
\begin{equation}
  \label{eq:1}
  b_{jk}=
  \#(\text{arrows from $\overset{j}{\bullet}$ to $\overset{k}{\bullet}$})
  -
  \#(\text{arrows from $\overset{k}{\bullet}$ to $\overset{j}{\bullet}$}) ,
\end{equation}
which gives 
the  skew-symmetric exchange matrix~$\mathbf{B}$,
\begin{equation}
  \label{B_matrix}
  \mathbf{B}  =
  ( b_{jk})=
  \begin{pmatrix}
    0 & 0 & -1 & -1 & 1 & 1\\
    0 & 0 & 1 & 1& -1 & -1 \\
    1 & -1 & 0 & 0 & -1 & 1 \\
    1 & -1 & 0 & 0 & 1 & -1 \\
    -1 & 1 & 1 & -1 & 0 & 0 \\
    -1 & 1 & -1 & 1 & 0 & 0
  \end{pmatrix} .
\end{equation}
With this exchange matrix~$\mathbf{B}$,
the log-canonical coordinates $y_i$ for $1\leq i \leq 6$ assigned to the
vertex~$\overset{i}{\bullet}$ in Fig.~\ref{fig:triangle_Sigma12}
have the Poisson algebra
\begin{equation}
  \label{eq:23}
  \{ y_j , y_k \} = b_{jk} y_j \, y_k, 
\end{equation}
Then
following the trace map~\cite{BonaWong11a}
(see~\cite{KHikami17a} for $\Sigma_{1,1}$ and $\Sigma_{0,4}$),
we can put
\begin{subequations}
  \label{k_from_cluster}
\begin{align}
  \label{eq:24}
%  \mathbb{y}_1
  \mathbb{k}_1
  & \mapsto
    \sqrt{y_2 y_3 y_5} + \sqrt{\frac{y_2 y_5}{y_3}}
    +\sqrt{\frac{y_5}{y_2 y_3}}
    +\frac{1}{\sqrt{y_2 y_3 y_5}},
\\    
  %\mathbb{x}_2
  \mathbb{k}_2
  & \mapsto
    \sqrt{y_1 y_3 y_6} + \sqrt{\frac{y_1 y_3}{y_6}}
    +\sqrt{\frac{y_3}{y_1 y_6}}
    +\frac{1}{\sqrt{y_1 y_3 y_6}},
  \\
%  \mathbb{y}_2
  \mathbb{k}_3
  & \mapsto
    \sqrt{y_2 y_4 y_6} + \sqrt{\frac{y_2 y_6}{y_4}}
    +\sqrt{\frac{y_6}{y_2 y_4}}
    +\frac{1}{\sqrt{y_2 y_4 y_6}},
  \\
%  \mathbb{x}_1
  \mathbb{k}_{1,2,3}
  & \mapsto
    \sqrt{y_1 y_4 y_5} + \sqrt{\frac{y_1 y_4}{y_5}}
    +\sqrt{\frac{y_4}{y_1 y_5}}
    +\frac{1}{\sqrt{y_1 y_4 y_5}},
  \\
  % \mathbb{z}_2
  \mathbb{k}_{1,2}
  & \mapsto
    \sqrt{y_1 y_2 y_5 y_6}
    +\sqrt{\frac{y_1 y_5 y_6}{y_2}}
    +\sqrt{\frac{y_1 y_5}{y_2 y_6}}+\sqrt{\frac{y_1 y_6}{y_2 y_5}}
    +\sqrt{\frac{y_1}{y_2 y_5 y_6}}+\frac{1}{\sqrt{y_1 y_2 y_5 y_6}},
  \\
%  \mathbb{z}_1
  \mathbb{k}_{2,3}
  & \mapsto
    \sqrt{y_1 y_2 y_3 y_4}
    +\sqrt{\frac{y_2 y_3 y_4}{y_1}}
    +\sqrt{\frac{y_2 y_3}{y_1 y_4}}+\sqrt{\frac{y_2 y_4}{y_1 y_3}}
    +\sqrt{\frac{y_2}{y_1 y_3 y_4}}+\frac{1}{\sqrt{y_1 y_2 y_3 y_4}},
  \\
  \mathbb{b}_1
  & \mapsto
    \sqrt{y_3 y_4 y_5y_6}+\frac{1}{\sqrt{y_3 y_4 y_5 y_6}},
  \\
  \mathbb{b}_2
  & \mapsto
    y_1 y_2 \sqrt{y_3 y_4 y_5 y_6}
    +\frac{1}{y_1 y_2 \sqrt{y_3 y_4 y_5 y_6}} .
\end{align}
\end{subequations}

The classical limit of the  Dehn twists in~\eqref{Mod_12} can be read as
\begin{subequations}
  \label{classical_Dehn}
  \begin{align}
    \mathscr{D}_1^{cl}:
    &
      \begin{pmatrix}
        \mathbb{k}_1 \\ \mathbb{k}_2 \\ \mathbb{k}_3 \\
        \mathbb{k}_{1,2} \\ \mathbb{k}_{2,3} \\ \mathbb{k}_{1,2,3}
      \end{pmatrix}
      \mapsto
      \begin{pmatrix}
        \mathbb{k}_1 \\ \mathbb{k}_1\mathbb{k}_2- \mathbb{k}_{1,2} \\ \mathbb{k}_3 \\
        \mathbb{k}_{2} \\ \mathbb{k}_1\mathbb{k}_{2,3}-\mathbb{k}_{1,2,3} \\ \mathbb{k}_{2,3}
      \end{pmatrix}
      ,      
      % \begin{pmatrix}
      %   \mathbb{x}_1 \\ \mathbb{x}_2 \\ \mathbb{y}_1 \\
      %   \mathbb{y}_2 \\ \mathbb{z}_1 \\ \mathbb{z}_2
      % \end{pmatrix}
      % \mapsto
      % \begin{pmatrix}
      %   \mathbb{z}_1 \\ \mathbb{y}_1\mathbb{x}_2-\mathbb{z}_2 \\
      %   \mathbb{y}_1 \\ \mathbb{y}_2 \\
      %   \mathbb{y}_1\mathbb{z}_1-\mathbb{x}_1 \\
      %   \mathbb{x}_2
      % \end{pmatrix}
      % ,
    \\
    \mathscr{D}_2^{cl}:
    &
      \begin{pmatrix}
        \mathbb{k}_1 \\ \mathbb{k}_2 \\ \mathbb{k}_3 \\
        \mathbb{k}_{1,2} \\ \mathbb{k}_{2,3} \\ \mathbb{k}_{1,2,3}
      \end{pmatrix}
      \mapsto
      \begin{pmatrix}
        \mathbb{k}_{1,2} \\ \mathbb{k}_2 \\
        \mathbb{k}_2 \mathbb{k}_3 -    \mathbb{k}_{2,3}\\
        \mathbb{k}_2 \mathbb{k}_{1,2}- \mathbb{k}_1 \\
        \mathbb{k}_{3} \\ \mathbb{k}_{1,2,3}
      \end{pmatrix}
      ,      
      % \begin{pmatrix}
      %   \mathbb{x}_1 \\ \mathbb{x}_2 \\ \mathbb{y}_1 \\
      %   \mathbb{y}_2 \\ \mathbb{z}_1 \\ \mathbb{z}_2
      % \end{pmatrix}
      %   \mapsto
      % \begin{pmatrix}
      %   \mathbb{x}_1 \\ \mathbb{x}_2 \\
      %   \mathbb{z}_2 \\  \mathbb{x}_2\mathbb{y}_2 -\mathbb{z}_1\\
      %   \mathbb{y}_2 \\
      %   \mathbb{x}_2\mathbb{z}_2-\mathbb{y}_1
      % \end{pmatrix}
      % ,
      % ,
    \\
    \mathscr{D}_3^{cl}:
    &
      \begin{pmatrix}
        \mathbb{k}_1 \\ \mathbb{k}_2 \\ \mathbb{k}_3 \\
        \mathbb{k}_{1,2} \\ \mathbb{k}_{2,3} \\ \mathbb{k}_{1,2,3}
      \end{pmatrix}
      \mapsto
      \begin{pmatrix}
        \mathbb{k}_1 \\ \mathbb{k}_{2,3} \\ \mathbb{k}_3 \\
        \mathbb{k}_{1,2,3} \\
        \mathbb{k}_3\mathbb{k}_{2,3}-\mathbb{k}_2 \\
        \mathbb{k}_3\mathbb{k}_{1,2,3} - \mathbb{k}_{1,2}
      \end{pmatrix} ,      
      % \begin{pmatrix}
      %   \mathbb{x}_1 \\ \mathbb{x}_2 \\ \mathbb{y}_1 \\
      %   \mathbb{y}_2 \\ \mathbb{z}_1 \\ \mathbb{z}_2
      % \end{pmatrix}
      %   \mapsto
      % \begin{pmatrix}
      %   \mathbb{x}_1\mathbb{y}_2-\mathbb{z}_2 \\ \mathbb{z}_1 \\
      %   \mathbb{y}_1 \\ \mathbb{y}_2 \\
      %   \mathbb{y}_2\mathbb{z}_1 -\mathbb{x}_2\\
      %   \mathbb{x}_1
      % \end{pmatrix}
      % ,
  \end{align}
\end{subequations}
where 
the boundary curves
$\mathbb{b}_1$ and $\mathbb{b}_2$ stay invariant
under all maps~$\mathscr{D}_a^{cl}$.
We can check that they satisfy the relations in~\eqref{Mod_12} as
\begin{equation}
  \label{mapping_class_group_C}
  \mathscr{D}_{1,3}^{cl}=
  \mathscr{D}_{3,1}^{cl},
  \qquad
  \mathscr{D}_{1,2,1}^{cl}
  =
  \mathscr{D}_{2,1,2}^{cl},
  \qquad
  \mathscr{D}_{2,3,2}^{cl}
  =
  \mathscr{D}_{3,2,3}^{cl},
  \qquad
  \left(\mathscr{D}_{1,2,3}^{cl}\right)^4
  =1 ,
\end{equation}
where
$\mathscr{D}_{i,j,\dots,k}^{cl}
=\mathscr{D}_i^{cl} \mathscr{D}_j^{cl}
\dots
\mathscr{D}_k^{cl} $ as before.
In view of the trace map~\eqref{k_from_cluster},
the automorphisms~\eqref{classical_Dehn} can be realized in terms of
the cluster mutation.
The mutation $\mu_k$ at the vertex $k$
is an involutive action  defined by
\begin{equation}
  \label{eq:32}
  \mu_k(\mathbf{y}, \mathbf{B})=
  \left( \widetilde{\mathbf{y}}, \widetilde{\mathbf{B}}
  \right) ,
\end{equation}
where
$\mathbf{y}=(y_k)$,
$\widetilde{\mathbf{y}}
=
\left( \widetilde{y}_k \right)$,
$\widetilde{\mathbf{B}}=
\left(
  \widetilde{b}_{ij} \right)$,
and
\begin{align}
  \widetilde{y}_i
  & =
    \begin{cases}
      y_k^{-1} , & \text{for $i=k$} ,
      \\
      y_i \left( 1+y_k^{-1} \right)^{-b_{ki}},
                 & \text{for $i\neq k$, $b_{ki}\geq 0$},
      \\
      y_i \left( 1+y_k \right)^{-b_{ki}} ,
                 &\text{for $i\neq k$, $b_{ki}\leq 0$} ,
    \end{cases}
  \\
  \widetilde{b}_{ij}
  & =
    \begin{cases}
      - b_{ij}, & \text{for $i=k$ or $j=k$},
      \\
      b_{ij}+\frac{1}{2}\left(|b_{ik}| b_{kj}+b_{ik}|b_{kj}|\right),
                & \text{otherwise}.
    \end{cases}
\end{align}
% We define
% the map
% \begin{equation}
%   \label{eq:34}
%   \Phi:
%   \Sk^{cl}(\Sigma_{1,2})
%   \to
%   \mathbb{C}\left[
%     y_1^{\pm\frac{1}{2}},
%     \dots,
%     y_6^{\pm \frac{1}{2}}
%   \right] ,
% \end{equation}
We then find that the Dehn twists~\eqref{classical_Dehn} can be induced from the
following mutations on~\eqref{k_from_cluster};
\begin{subequations}
  \begin{align}
    \label{eq:31}
    \mathscr{D}_1^{cl}=
    \sigma_{3,5} \, \mu_{2,3,2}:
    &
      \begin{pmatrix}
        y_1 \\ y_2 \\ y_3 \\ y_4 \\ y_5 \\ y_6
      \end{pmatrix}
      \mapsto
      \begin{pmatrix}
        \frac{y_1 y_2 y_3}{1 + y_2 + y_2 y_3}\\
        \frac{1 + y_2 + y_2 y_3}{y_3}\\
        \frac{1}{
        y_2(1 +  y_3)}
        \\
        \frac{y_2 (1 + y_3) y_4}{1 + y_2 +  y_2 y_3}
        \\
        (1 + y_3) (1 + y_2 + y_2 y_3) y_5
        \\
        \frac{y_3 y_6}{1 + y_3}
      \end{pmatrix}
      ,
    \\
    \mathscr{D}_2^{cl}=
    \sigma_{3,6} \, \mu_{1,6,1}:
    &
      \begin{pmatrix}
        y_1 \\ y_2 \\ y_3 \\ y_4 \\ y_5 \\ y_6
      \end{pmatrix}
      \mapsto
      \begin{pmatrix}
        \frac{1 + y_1 + y_1 y_6}{y_6}
        \\
        \frac{y_1 y_2 y_6}{1 + y_1 + y_1 y_6}
        \\
        \frac{1}{y_1(1+ y_6)}
        \\
        \frac{y_4 y_6}{1 + y_6}
        \\
        \frac{y_1 y_5 (1 + y_6)}{1 + y_1 +  y_1 y_6}
        \\
        y_3 (1 + y_6) (1 + y_1 + y_1 y_6)
      \end{pmatrix}
      ,
    \\ 
    \mathscr{D}_3^{cl}=
    \sigma_{4,6} \, \mu_{2,4,2}:
    &
      \begin{pmatrix}
        y_1 \\ y_2 \\ y_3 \\ y_4 \\ y_5 \\ y_6
      \end{pmatrix}
      \mapsto
      \begin{pmatrix}
        \frac{y_1 y_2 y_4}{1 + y_2 + y_2 y_4}
        \\
        \frac{1 + y_2 +  y_2 y_4}{y_4}
        \\
        \frac{y_2 y_3 (1 + y_4)}{1 + y_2 +  y_2 y_4}
        \\
        (1 + y_4) (1 + y_2 + y_2 y_4) y_6
        \\
        \frac{y_4 y_5}{1 + y_4}
        \\
        \frac{1}{y_2(1 +   y_4)}
      \end{pmatrix}
      ,
  \end{align}
\end{subequations}
where $\sigma_{i,j}$ means a permutation
relabeling of the vertices $i$ and $j$, and
$\mu_{i,j,\dots,k}=\mu_i \mu_j\dots \mu_k$.
See that
the exchange matrix $\mathbf{B}$~\eqref{B_matrix} is invariant under
the mutations $\mathscr{D}_a^{cl}$.
% We have
% \begin{equation}
%   \label{mapping_class_group_C}
%   \mathscr{M}_{1,3}^{cl}
%   = \mathscr{M}_{3,1}^{cl},
%   \qquad
%   \mathscr{M}_{1,2,1}^{cl}
%   =
%   \mathscr{M}_{2,1,2}^{cl},
%   \qquad
%   \mathscr{M}_{2,3,2}^{cl}
%   =
%   \mathscr{M}_{3,2,3}^{cl},
%   \qquad
%   \left(\mathscr{M}_{1,2,3}^{cl}\right)^4
%   =1 ,
% \end{equation}
% where
% $\mathscr{M}_{i,j,\dots,k}^{cl}
% =\mathscr{M}_i^{cl} \mathscr{M}_j^{cl}
% \dots
% \mathscr{M}_k^{cl} $.
% We  thus have a commutative diagram,
% \begin{equation}
%   \begin{tikzcd}
%     \Sk^{cl}(\Sigma_{1,2})\arrow[r, "\mathscr{D}_i^{cl}"]
%     \arrow[d, "\Phi"']
%     & \Sk^{cl}(\Sigma_{1,2})
%     \arrow[d, "\Phi"]
%     \\
%     \mathbb{C}\left[
%       y_1^{\pm\frac{1}{2}}, \dots, y_6^{\pm\frac{1}{2}}
%     \right]
%     \arrow[r, "\mathscr{M}_i^{cl}"']
%     &
%     \mathbb{C}\left[
%       y_1^{\pm\frac{1}{2}}, \dots, y_6^{\pm\frac{1}{2}}
%     \right]
%   \end{tikzcd}
% \end{equation}
We note that
the last identity in~\eqref{mapping_class_group_C} gives the mutation
loop of length~$30$,
\begin{equation}
  \mu_{
    2,5,2,1,5,1,2,5,3,2,1,3,1,2,3,
    6,2,1,6,1,2,6,4,2,1,4,1,2,4,2}
  =1 .
\end{equation}
See~\cite{KimYam20a} where the mutation loop of length~$32$ is given.
% of the cluster algebra
% for~\eqref{B_matrix}.

% The Dehn twist about the curve $\mathbb{k}$ is
% \begin{equation}
%   \label{eq:29}
%   \mathscr{D}_{\mathbb{k}}^{cl}(\mathbb{x})
%   =
%   \left\{\mathbb{k} ~,~ \mathbb{x}\right\}
%   +\frac{1}{2} \mathbb{k} \, \mathbb{x} .
% \end{equation}

%%%%%%%%%%%%%%%%%%%
\section{Concluding Remarks}
We have constructed the generalized DAHA for the skein algebra $\Sk_A(\Sigma_{1,2})$.
It remains for future to establish
a relationship between the DAHA
operators for closed curve~$\mathbb{k}$
(or, the DAHA polynomial
$\mathcal{A}(\mathbb{k})(1)$ and its $N$-colored generalization)
and the quantum group $U_q(sl_2)$ knot invariants.
As the Riemann surface $\Sigma_{g,n}$ can be naturally
decomposed into
$\Sigma_{1,1}$ and $\Sigma_{1,2}$,
we expect that
our results on the twice-punctured surface will be useful for the generalized DAHA for
the skein algebra on
$\Sigma_{g,n}$.
We  hope that our DAHA realization for $\Sk_A(\Sigma_{1,2})$ may be useful in
studies of
$(1,1)$-knots
(see, \emph{e.g.},~\cite{CattaMulaz04a}).

%%%%%%%%%%%%%%%%%%%%%%%%%%
\section*{Acknowledgments}
The work of KH is supported in part by
% Grants-in-Aid for Scientic Research
JSPS KAKENHI Grant Numbers
JP23K22388,
JP22H01117,
JP20K03931.

%%%%%%%%%%%%%% % newpage
% \bibliographystyle{halphaKH}
% % %%%%%%%%%%%%%%%%%%
% \bibliography{_def,gravity,gravity2,square,math,ba,tba,math5,vm,square2,math4,qalg,math3,math2,poisson,geometry,soliton,cft,knot,tqft,comb,number,qhe}

\end{document}